\documentclass{article}

\usepackage{geometry}[margins=1inch]

\usepackage{ulem}
\usepackage[utf8]{inputenc}
\usepackage{amsthm}
\usepackage{amsmath}
\usepackage{amssymb}
\usepackage{amsfonts}
\usepackage{mathtools}
\usepackage[ruled,vlined]{algorithm2e}
\usepackage{braket}
\usepackage{graphicx}
\usepackage{xcolor}
\usepackage{xspace}

\newtheorem{lemma}{Lemma}
\newtheorem{theorem}{Theorem}
\newtheorem{definition}{Definition}

\newtheorem{proposition}{Proposition}
\newtheorem{corollary}{Corollary}

\newcount\Comments  
\definecolor{darkgreen}{rgb}{0,0.6,0}
\newcommand{\kibitz}[2]{\ifnum\Comments=1{\color{#1}{#2}}\fi}

\newcommand{\word}[1]{\ensuremath{\bar{#1}}}
\newcommand{\bracelet}[1]{\ensuremath{\hat{\mathbf{#1}}}}
\newcommand{\necklace}[1]{\ensuremath{\tilde{\mathbf{#1}}}}

\newcount\Sketches 
\newcommand{\ProofSketch}[1]{\ifnum\Sketches=1{ \noindent {\bf Proof Sketch.}#1 \hfill \qedsymbol}\fi}

\title{Ranking Bracelets in Polynomial Time \footnote{A preliminary conference version of this work appeared in the Proceedings of the 32nd Annual Symposium on Combinatorial Pattern Matching, CPM 2021 \cite{Adamson2021}.}}

\author{Duncan Adamson, Argyrios Deligkas, Vladimir V. Gusev, Igor Potapov}

\date{\today}

\begin{document}

\maketitle

\begin{abstract}
    The main result of the paper is the first  polynomial-time algorithm for ranking bracelets. The time-complexity of the algorithm is $O(k^2\cdot n^4)$, where $k$ is the size of the alphabet and $n$ is the length of the considered bracelets.
    The key part of the algorithm is to compute the rank 
    of any word with respect to the set of bracelets by finding three other ranks: the rank over all necklaces, the rank over palindromic necklaces, and the rank over enclosing apalindromic necklaces. 
    The last two concepts are introduced in this paper.
    These ranks are key components to our algorithm in order to decompose the problem into parts.
    Additionally, this ranking procedure is used to build a polynomial-time unranking algorithm.
\end{abstract}

\section{Introduction}
Counting, ordering, and generating basic discrete structures such as strings, permutations, set-partitions, etc. are 
fundamental tasks in computer science. A variety of such algorithms are assembled in the fourth volume of the prominent
series ``The art of computer programming'' by D.~Knuth~\cite{Knuth4a}. Nevertheless, this research direction remains 
very active~\cite{Mutze20}.

If the structures under consideration are linearly ordered, e.g. a set of words under the dictionary (lexicographic) order, then a unique integer can be assigned to every structure.
The \emph{rank} (or index) of a structure is the number of structures that are smaller than it.
The {\em ranking} problem asks to compute the rank of a given structure, while the {\em unranking} problem corresponds to its reverse: compute the structure of a given rank. Ranking has been studied for various objects including partitions~\cite{RankPartition}, permutations~\cite{RankPerm1,RankPerm2}, combinations~\cite{RankComb}, etc.
Unranking has similarly been studied for objects such as permutations~\cite{RankPerm2} and trees~\cite{Gupta1983,Pallo1986}.

Both ranking and unranking are straightforward for the set of all words over a finite alphabet (assuming the standard lexicographic order), but 
they immediately cease to be so, as soon as additional symmetry is introduced.
One of such examples is a class of necklaces \cite{Graham1994}.
A \emph{necklace}, also known as a \emph{cyclic word}, is an equivalence class of all words under the cyclic rotation operation, also known as a cyclic shift.
Necklaces are classical combinatorial objects and they remain an object of study in other contexts such as total search problems~\cite{SplitNeck19} or circular splicing systems~\cite{CircularSplicing}.

The \emph{rank} of a word $w$ for a given set $S$ and its ordering is the number of words in $S$ that are smaller than $w$.
Often the set is a class of words, for instance all words of a given length over some alphabet.
The first class of cyclic words to be ranked were \emph{Lyndon words} - fixed length aperiodic cyclic words - by Kociumaka et. al. \cite{Kociumaka2014} who provided an $O(n^3)$ time algorithm.
An algorithm for ranking necklaces - fixed length cyclic words - was given by Kopparty et. al. \cite{Kopparty2016}, without tight bounds on the complexity.
A quadratic algorithm for ranking necklaces was provided by Sawada et al. \cite{Sawada2017}.

\begin{figure}
    \centering
    \begin{tabular}{l l l l l l l l l l}
        1. & aaaaaaaa & 7. & aaaababb & 13. & aaabbabb & 19. & aababbbb & 25. &ababbabb\\
        2. & aaaaaaab & 8. & aaaabbbb & 14. & aaabbbbb & 20. & aabbaabb & 26. &ababbbbb\\
        3. & aaaaaabb & 9. & aaabaaab & 15. & aabaabab & 21. & aabbabbb & 27, &abbabbbb\\
        4. & aaaaabab & 10. & aaabaabb & 16. & aabaabbb & 22. & aabbbbbb & 28. &abbbabbb\\
        5. & aaaaabbb & 11. & aaababab & 17. & aabababb & 23. & abababab & 29. & abbbbbbb\\
        6. & aaaabaab & 12. & aaabbabb & 18. & aababbab & 24. & abababbb & 30. & bbbbbbbb
    \end{tabular}
    \caption{List of all bracelets of length $8$ over the alphabet $\{a,b\}$.}
    \label{fig:rank_example}
\end{figure}

This paper answers the open problem of ranking \emph{bracelets}, posed by Sawada and Williams \cite{Sawada2017}.
Bracelets are necklaces that are minimal under both \emph{cyclic shifts} and \emph{reflections}.
Figure \ref{fig:rank_example} provides an example of the ranks of length 8 bracelets over a binary alphabet.
Bracelets have been studied extensively, with results for counting and generation in both the normal and fixed content cases \cite{Karim2013,Sawada2001}.


This paper presents the first algorithm for ranking bracelets
of length $n$ over an alphabet of size $k$
in polynomial time, with a time complexity of $O(k^2 \cdot n^{4})$.
This algorithm is further used to unrank bracelets in $O(n^5 \cdot k^2 \cdot \log(k))$. time.
These polynomial time algorithms improve upon the exponential time brute-force algorithm.

We briefly mention our additional interest to this problem.
Combinatorial necklaces and bracelets provide discrete representation of periodic motives in crystals.
The problems on finding diverse and representative samples of languages of  necklaces and bracelets can speed up space exploration in crystal structures \cite{Collins17}.
The essential component for building 
representative sample require efficient 
procedures for ranking bracelets.

\section{Preliminaries}
\label{sec:prelims}

\subsection{Definitions and Notation}
\label{subsec:definitions_notation}

Let $\Sigma$ be a finite alphabet.
We denote by $\Sigma^*$ the set of all words over $\Sigma$ and by $\Sigma^n$ the set of all words of length $n$.
For the remainder of this paper, let $k = |\Sigma|$.
The notation $\word{w}$ is used to clearly denote that the variable $w$ is a word.
The length of a word $\word{u} \in \Sigma^*$ is denoted $|\word{u}|$.
We use $\word{u}_i$, for any $i \in 1\hdots|\word{u}|$ to denote the $i^{th}$ symbol of $\word{u}$.
The \emph{reversal} operation on a word $\word{w} = \word{w}_1 \word{w}_2 \hdots \word{w}_n$, denoted  by $\word{w}^{R}$, returns the word $\word{w}_n \hdots \word{w}_2 \word{w}_1$.

In the present paper we assume that $\Sigma$ is linearly ordered.
Let $[n]$ return the ordered set of integers from $1$ to $n$ inclusive.
Given 2 words $\word{u},\word{v} \in \Sigma^*$ where $|\word{u}| \leq |\word{v}|$, $\word{u} = \word{v}$ if and only if $|\word{u}| = |\word{v}|$ and $\word{u}_i = \word{v}_i$ for every $i \in [|\word{u}|]$.
A word $\word{u}$ is \emph{lexicographically smaller} than $\word{v}$ if there exists an $i \in [|\word{u}|]$ such that $\word{u}_1 \word{u}_2 \hdots \word{u}_{i-1} = \word{v}_1 \word{v}_2 \hdots \word{v}_{i-1}$ and $\word{u}_i < \word{v}_i$.
For example, given the alphabet $\Sigma = \{a,b\}$ where $a < b$, the word $aaaba$ is smaller than $aabaa$ as the first 2 symbols are the same and $a$ is smaller than $b$.
For a given set of words $\mathbf{S}$, the rank of $\word{v}$ with respect to $\mathbf{S}$ is the number of words in $\mathbf{S}$ that are smaller than $\word{v}$.

The {\em rotation} of a word $\word{w}=\word{w}_1 \word{w}_2 \hdots \word{w}_n$ by $r \in [n-1]$ returns the word $\word{w}_{r + 1} \hdots \word{w}_n \word{w}_1 \hdots \word{w}_r$, and is denoted by  $\langle \word{w} \rangle_r$, i.e. $\langle \word{w}_1 \word{w}_2 \hdots \word{w}_n \rangle_r = \word{w}_{r + 1} \hdots \word{w}_n \word{w}_1 \hdots \word{w}_r$.
Under the rotation operation, $\word{u}$ is equivalent to $\word{v}$ if $\word{v} = \langle \word{w} \rangle_r$ for some $r$.
The $t^{th}$ power of a word $\word{w} = \word{w}_1 \hdots \word{w}_n$, denoted $\word{w}^t$, is equal to $\word{w}$ repeated $t$ times.
For example $(aab)^3 = aabaabaab$.
A word $\word{w}$ is \emph{periodic} if there is some word $\word{u}$ and integer $t \geq 2$ such that $\word{u}^t = \word{w}$.
Equivalently, word $\word{w}$ is \emph{periodic} if there exists some rotation $0 < r < |\word{w}|$ where $\word{w} = \langle \word{w}\rangle_r$.
A word is \emph{aperiodic} if it is not periodic.
The \emph{period} of a word $\word{w}$ is the length of the smallest word $\word{u}$ for which there exists some value $t$ for which $\word{w} = \word{u}^t$.

A \emph{cyclic word}, also called a \emph{necklace}, is the equivalence class of words under the rotation operation.
For notation, a word $\word{w}$ is written as $\necklace{w}$ when treated as a necklace.
Given a necklace $\necklace{w}$, the \emph{necklace representative} is the lexicographically smallest element of the set of words in the equivalence class $\necklace{w}$.
The necklace representative of $\necklace{w}$ is denoted $\langle \necklace{w} \rangle$, and the $r^{th}$ shift of the necklace representative is denoted $\langle\necklace{w}\rangle_{r}$.
The reversal operation on a necklace $\necklace{w}$ returns the necklace $\necklace{w}^R$ containing the reversal of every word $\word{u} \in \necklace{w}$, i.e. $\necklace{w}^R = \{\word{u}^R: \word{u} \in \necklace{w}\}$.
Given a word $\word{w}$, $\langle \word{w} \rangle$ will denote the necklace representative of the necklace containing $\word{w}$, i.e. the representative of $\necklace{u}$ where $\word{w} \in \necklace{u}$.

A \emph{subword} of the cyclic word $\word{w}$, denoted $\word{w}_{[i,j]}$ is the word $\word{u}$ of length $|\word{w}| + j - i - 1 \bmod |\word{w}||)$ such that $\word{u}_{a} = \word{w}_{i - 1 + a \bmod |\word{w}|}$.
For notation $\word{u} \sqsubseteq \word{w}$ denotes that $\word{u}$ is a subword of $\word{w}$.
Further, $\word{u} \sqsubseteq_{i} \word{w}$ denotes that $\word{u}$ is a subword of $\word{w}$ of length $i$.
If $\word{w} = \word{u} \word{v}$, then $\word{u}$ is a prefix and $\word{v}$ is a suffix.
A prefix or suffix of a word $\word{u}$ is \emph{proper} if its length is smaller than $|\word{u}|$.
For notation, the tuple $\mathbf{S}(\word{v},\ell)$ is defined as the set of all subwords of $\word{v}$ of length $\ell$.
Formally let $\mathbf{S}(\word{v},\ell) = \{\word{s} \sqsubseteq \word{v}: |\word{s}| = \ell\}$.
Further, $\mathbf{S}(\word{v},\ell)$ is assumed to be in lexicographic order, i.e. $\mathbf{S}(\word{v},\ell)_1 \geq \mathbf{S}(\word{v},\ell)_2 \geq \hdots \mathbf{S}(\word{v},\ell)_{|\word{v}|}$. 

A \emph{bracelet} is the equivalence class of words under the combination of the rotation and the reversal operations.
In this way a bracelet can be thought of as the union of two necklace classes $\necklace{w}$ and $\necklace{w}^R$, hence $\bracelet{w} = \necklace{w}\cup \necklace{w}^R$.
Given a bracelet $\bracelet{w}$, the \emph{bracelet representative} of $\bracelet{w}$, denoted by $[ \bracelet{w} ]$, is the lexicographically smallest word $\word{u} \in \bracelet{w}$. 

A necklace $\necklace{w}$ is \emph{palindromic} if $\necklace{w} = \necklace{w}^R$. 
This means that the reflection of every word in $\necklace{w}$ is in $\necklace{w}^R$, i.e. given $\word{u} \in \necklace{w}, \word{u}^R \in \necklace{w}^R$.
Note that for any word $\word{w} \in \necklace{a}$, where $\necklace{a}$ is a palindromic necklace, either $\word{w} = \word{w}^R$, or there exists some rotation $i$ for which $\langle \word{w} \rangle_i = \word{w}^R$.

Let $\necklace{u}$ and $\necklace{v}$ be a pair of necklaces belonging to the same bracelet class.
For simplicity assume that $\langle \necklace{u} \rangle < \langle \necklace{v} \rangle$. 
The bracelet $\bracelet{u}$  \emph{encloses} a word $\word{w}$ if $\langle \necklace{u} \rangle < \word{w} < \langle \necklace{v} \rangle$.
An example of this is the bracelet $\bracelet{u} = aabc$ which encloses the word $\word{w} = aaca$ as $aabc < aaca < aacb$.
The set of all bracelets which enclose $\word{w}$ are referred to as the set of bracelets \emph{enclosing} $\word{w}$.

\subsection{Bounding Subwords}
\label{subsec:bounding_subwords}
For both the palindromic and enclosing cases the number of necklaces smaller than $\word{v} \in \Sigma^n$ is computed by iteratively counting the number of words of length $n$ for which no subword is smaller than $\word{v}$. The set of such words, denoted by $\mathbf{S}_n$, 
will be analysed iteratively as well, since it can have an exponential size.
In order to relate $\mathbf{S}_i$ to $\mathbf{S}_{i+1}$, we will split $\mathbf{S}_i$ into parts 
using the positions of length $i$ subwords of $\word{v}$ with respect to the lexicographic order on $\mathbf{S}_i$.
Informally, every $\word{w} \in \mathbf{S}_i$ can be associated with the unique lower bound from $\mathbf{S}(\word{v}, i)$, which will be used to identify the parts leading us to the following definition.


\begin{definition}
Let $\word{w}, \word{v} \in \Sigma^*$ where $|\word{w}| \leq |\word{v}|$.
The word $\word{w}$ is  \emph{bounded} (resp. strictly bounded) by $\word{s} \sqsubseteq_{|\word{w}|} \word{v}$, if $\word{s} \leq \word{w}$ (resp. $\word{s} < \word{w}$) and there is no $\word{u} \sqsubseteq_{|\word{w}|} \word{v}$ such that $\word{s} < \word{u} \leq \word{w}$.
\label{def:bounding}
\end{definition}

\noindent
The aforementioned parts $\mathbf{S}_i(\word{s})$ contain all words $\word{w} \in \mathbf{S}_i$ such that $\word{s} \sqsubseteq_{|\word{w}|} \word{v}$. 
The key observation is that words of the form $x\word{w}$ for all $\word{w} \in \mathbf{S}_i$ and some fixed symbol $x \in \Sigma$ belong to the
same set $\mathbf{S}_{i+1}(\word{s}')$, where $\word{s}' \sqsubseteq \word{v}$.
The same holds true for words of the form $\word{w}x$.
Thus, we can compute the corresponding $\word{s}'$ for all pairs of $\word{s}$ and $x$ in order to derive sizes of $\mathbf{S}_{i+1}(\word{s}')$.
Moreover, this relation between $\word{s}$, $x$ and $\word{s}'$ is independent of $i$ allowing us to store this information in two $n^2 \times k$ arrays $XW$ and $WX$. Both arrays will be indexed by the words $\word{s} \sqsubseteq \word{v}$ and characters $x \in \Sigma$.
Given a word $\word{w}$ strictly bounded by $\word{s}$, $XW[\word{s},x]$ will contain the word $\word{s}' \sqsubseteq_{|\word{s}| + 1} \word{v}$ strictly bounding $x \word{w}$.
Similarly, $WX[\word{s},x]$ will contain the word $\word{s}' \sqsubseteq_{|\word{s}| + 1} \word{v}$ strictly bounding $\word{w} x$.
By precomputing these arrays, the cost of determining these words can be avoided during the ranking process.
In order to compute these arrays, the following technical Lemmas are needed.

\begin{lemma}
Let $\word{w},\word{v}  \in \Sigma^*$, $|\word{w}| < |\word{v}|$,  let $x \in \Sigma$ and let $\word{s} \sqsubseteq_{|\word{w}|} \word{v}$ be the subword of $\word{v}$ that bounds $\word{w}$.
The word $\word{s}' \sqsubseteq \word{v}$ bounds $x \word{w}$ if and only if $\word{s}'$ bounds $x \word{s}$.
\label{lem:bounding_same_xw}
\end{lemma}

\begin{proof}
Let $\word{s}' \sqsubseteq \word{v}$ bound $x \word{w}$.
Since $\word{s} \leq \word{w}$, we have $x \word{s} \leq x \word{w}$.
For the sake of contradiction assume that $x \word{s}$ is bounded by $\word{u} < \word{s}'$.
If $\word{u}_1 < x$ then $\word{s}'_1 = x$ as for any smaller value of $\word{s}'_1$, $\word{u}$ would not bound $x \word{w}$.
Under this assumption $\word{s} < \word{s}'_{[2,|\word{s}'|]} \leq \word{w}$, in which case $\word{s}'_{[2,|\word{s}'|]}$ would bound $\word{w}$, contradicting this assumption.
If $\word{u}_1 = x$, then again $\word{s} < \word{s}'_{[2,|\word{s}'|]} < \word{w}$, in which case $\word{s}'_{[2,|\word{s}'|]}$ bounds $\word{w}$ contradicting the original assumption that $\word{s}$ bounds $\word{w}$.

In the other direction, let $\word{s}'$ bound $x \word{s}$.
If $\word{s}'$ does not bound $x \word{w}$ then there must exist some word $\word{u}$ bounding $x \word{w}$.
As $x \word{s} < \word{u} < x \word{w}$, $\word{u}_1 = x$ hence $\word{u} = x \word{u}'$.
Therefore $\word{u}'$ bounds $\word{w}$, contradicting our original assumption.
Hence $\word{s}'$ bounds $x \word{w}$ if and only if $\word{s}'$ bounds $x \word{s}$ where $\word{s}$ bounds $\word{w}$.
\end{proof}

\begin{lemma}
Let $\word{w},\word{v} \in \Sigma^*$,  let $x \in \Sigma$ and let $\word{s} \sqsubseteq \word{v}$ be the subword of $\word{v}$ that bounds $\word{w}$.
Let $\word{s}' \sqsubseteq \word{v}$ bound $\word{w} x$.
Either $\word{s}'$ bounds $\word{s} x$, or $\word{s}' = \word{s} y$ for $y > x$.
\label{lem:wx_part_one}
\end{lemma}

\begin{proof}
Let $\word{u}$ bound $\word{s} x$.
If $\word{u} \neq \word{s}'$ then as $\word{w} x \geq \word{s}' > \word{s} x \geq \word{u}$, if $\word{s}'_{[1,|\word{\word{s}}|]} > \word{s}$ then $\word{s}'_{[1,|\word{\word{s}}|]}$ must bound $\word{w}$, contradicting the assumption that $\word{s}$ bounds $\word{w}$.
Therefore the only possible value of $\word{s}' > \word{s} x$ is when $\word{s}' = \word{s} y$ for some $y > x$.
\end{proof}

\begin{lemma}
Let $\word{w}, \word{u}, \word{v} \in \Sigma^*$, let $x \in \Sigma$ and let $\word{s} \sqsubseteq \word{v}$ be the subword of $\word{v}$ that strictly bounds both $\word{w}$ and $\word{u}$.
The word $\word{s}' \sqsubseteq \word{v}$ which bounds $\word{w} x$ will also bound $\word{u} x$.
\label{lem:bounding_same_wx}
\end{lemma}

\begin{proof}
For the sake of contradiction assume $\word{u} x$ is bounded by $\word{t}$.
This implies that $\word{s}' < \word{w} x < \word{t} \leq \word{u} x$.
Following Lemma \ref{lem:wx_part_one}, $\word{t} = \word{s} y$ for $y > x$.
However, as $\word{w} > \word{s}$, $\word{t}$ must be less than $\word{w}$ and hence $\word{t}$ would be a better bound for $\word{w}$.
\end{proof}

\begin{proposition}
Let $\word{v} \in \Sigma^n$.
The array $XW[\word{s} \sqsubseteq \word{v},x \in \Sigma]$ such that $XW[\word{s},x]$ strictly bounds $x \word{w}$ for every $\word{w}$ strictly bounded by $\word{s}$ can be computed in time $O(k \cdot n^3 \cdot  \log(n))$.
\label{prop:complexity_XW}
\end{proposition}

\begin{proof}
Given some pair of arguments $\word{s} \sqsubseteq \word{v}, x \in \Sigma$, the word bounding $x \word{s}$ can be found through a binary search on $\mathbf{S}(\word{v},|\word{s}| + 1)$. 
As each comparison will take at most $O(n)$ operations, and at most $\log(n)$ comparisons are needed, each entry can be computed in $O(n \log(n))$ operations.
As there are $O(n^2)$ subwords of $\word{v}$ and $k$ characters in $\Sigma$, there is at most $O(k \cdot n^3 \log(n))$ operations needed.
\end{proof}

\begin{proposition}
Let $\word{v} \in \Sigma^n$.
The array $WX[\word{s} \sqsubseteq \word{v},x \in \Sigma]$, such that $WX[\word{s},x]$ strictly bounds $\word{w} x$ for every $\word{w}$ strictly bounded by $\word{s}$, can be computed in $O(k \cdot n^3 \cdot \log(n))$ time.
\label{prop:complexity_WX}
\end{proposition}

\begin{proof}
For some pair pair of arguments $\word{s} \sqsubseteq \word{v}, x \in \Sigma$, let $\word{w}$ be the smallest word greater than $\word{s}$.
The word bounding $\word{w} x$ can be found through a binary search of $\mathbf{S}(\word{v}, |\word{s}| + 1)$.
Following Lemma \ref{lem:bounding_same_wx}, given any word $\word{u}$ strictly bounded by $\word{s}$, $\word{u} x$ will also be bounded by the same word bounding $\word{w} x$.
As in Proposition \ref{prop:complexity_XW}, each comparison will take at most $O(n)$ operations, with the search requiring at most $\log(n)$ comparisons.
As there are $n^2 \cdot k$ arguments, at most $O(k \cdot n^3 \log(n))$ operations are needed to compute every value of $WX$.
\end{proof}



\section{Ranking Bracelets}
\label{sec:ranking-bracelets}
The main result of the paper is the first algorithm for ranking bracelets.
In this paper, we tacitly assume that we are ranking a word $\word{v}$ of length $n$.
The time-complexity of the ranking algorithm is $O(k^2\cdot n^4)$, where $k$ is the size of the alphabet and $n$ is the length of the considered bracelets.
The key part of the algorithm is to compute the rank 
of the word $\word{v}$ with respect to the set of bracelets by finding three other ranks: the rank over all necklaces, the rank over palindromic necklaces, and the rank over enclosing apalindromic necklaces.

A bracelet can correspond to two apalindromic necklaces, or to exactly one palindromic necklace.
If a bracelet $\bracelet{b}$ corresponds to two necklaces $\necklace{l}_b$ and $\necklace{r}_b$, then it is important to take into account the lexicographical positions of these two necklaces $\necklace{l}_b$ and $\necklace{r}_b$ with respect to a given word $\word{v}$.
There are three possibilities: $\necklace{l}_b$ and $\necklace{r}_b$ could be less than $\word{v}$;  $\necklace{l}_b$ and $\necklace{r}_b$ encloses $\word{v}$, e.g. $\necklace{l}_b < \word{v} < \necklace{r}_b$, or both of necklaces $\necklace{l}_b$ and $\necklace{r}_b$ are greater than $\word{v}$. This is visualised in Figure \ref{fig:ranking_examples}.
Therefore the number of bracelets smaller than a given word $w$ can be calculated by adding the number of palindromic necklaces less than $\word{v}$, enclosing bracelets smaller than $\word{v}$ and half of all other apalindromic and non-enclosing necklaces smaller than $\word{v}$.
Let us define the following notation is used for the rank of $\word{v} \in \Sigma^n$ for sets of bracelets and necklaces.
\begin{itemize}
    \item[$\circ$] $RN(\word{v})$ denotes the rank of $\word{v}$ with respect to the set of \emph{necklaces} of length $n$ over $\Sigma$.
    \item[$\circ$] $RP(\word{v})$ denotes the rank of $\word{v}$ with respect to the set of \emph{palindromic necklaces} over $\Sigma$.
    \item[$\circ$] $RB(\word{v})$ denotes the rank of $\word{v}$ with respect to the set of \emph{bracelets} of length $n$ over $\Sigma$.
    \item[$\circ$] $RE(\word{v})$ denotes the rank of $\word{v}$ with respect to the set of \emph{bracelets enclosing} $\word{v}$.
\end{itemize}

\begin{figure}[ht]
    \centering
    \includegraphics[scale=0.85]{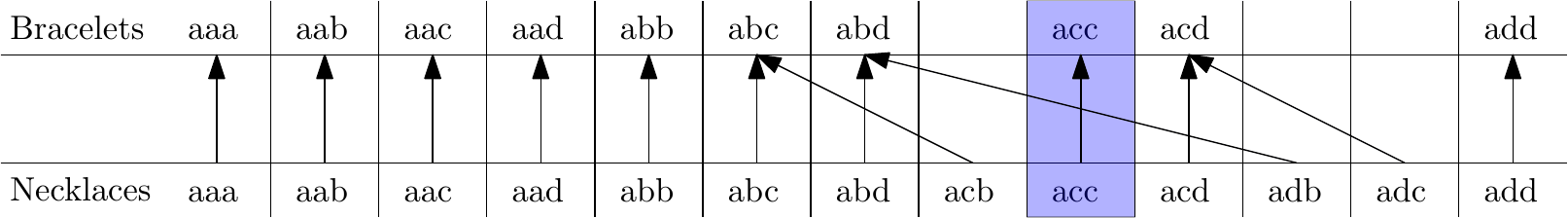}
    \caption{In this example the top line represents the set of bracelets and the bottom line the set of necklaces, with arrows indicated which necklace corresponds to which bracelet.
    Assuming we wish to rank the word $acc$ (highlighted), $abc$ and $acb$ are apalindromic necklaces smaller than $acc$, while $abd$ encloses $acc$.
    All other necklaces are palindromic.
    }
    \label{fig:ranking_examples}
\end{figure}

\noindent
In Lemma \ref{lem:ranking_bracelets} below, we show that $RB(\word{v})$ can be expressed via  $RN(\word{v})$, $RP(\word{v})$ and $RE(\word{v})$.
The problem of computing $RN(\word{v})$ has been solved in quadratic time \cite{Sawada2017}, so the goal of the paper is to design efficient procedures for computing  $RP(\word{v})$ and $RE(\word{v})$.

\begin{lemma}
The rank of a word $\word{v} \in \Sigma^n$ with respect to the set of bracelets of length $n$ over the alphabet $\Sigma$ is given by $RB(\word{v}) = \frac{1}{2}\left(RN(\word{v}) + RP(\word{v}) + RE(\word{v}) \right)$.
\label{lem:ranking_bracelets}
\end{lemma}

\begin{proof}
Simply dividing the number of necklaces by 2 will undercount the number of bracelets, while doing nothing will overcount.
Therefore to get the correct number of bracelets, those bracelets corresponding to only 1 necklace must be accounted for.
A bracelet $\bracelet{a}$ will correspond to 2 necklaces smaller than $\word{v}$ if and only if $\bracelet{a}$ does not enclose $\word{v}$ and $\bracelet{a}$ is apalindromic.
Therefore the number of bracelets corresponding to 2 necklaces is $\frac{1}{2}\left(RN(\word{v}) - RP(\word{v}) - RE(\word{v})\right)$.
The number of bracelets enclosing $\word{v}$ is equal to $RE(\word{v})$. 
The number of bracelets corresponding to palindromic necklaces is equal to $RP(\word{v})$.
Therefore the total number of bracelets is $\frac{1}{2}\left(RN(\word{v}) - RP(\word{v}) - RE(\word{v})\right) + RP(\word{v}) + RE(\word{v}) = \frac{1}{2}\left(RN(\word{v}) + RP(\word{v}) + RE(\word{v})\right)$.
\end{proof}

\noindent
Lemma \ref{lem:ranking_bracelets} provides the basis for ranking bracelets.
Theorem \ref{thm:ranking_complexity} uses Lemma \ref{lem:ranking_bracelets} to get the complexity of the ranking process.
The remainder of this paper will prove Theorem \ref{thm:ranking_complexity}, starting with the complexity of ranking among palindromic necklaces in Section \ref{sec:Ptv} followed by the complexity of ranking enclosing bracelets in Section \ref{sec:Stv}.

\begin{theorem}
Given a word $\word{v} \in \Sigma^n$, the rank of $\word{v}$ with respect to the set of bracelets of length $n$ over the alphabet $\Sigma$, $RB(\word{v})$, can be computed in $O(k \cdot n^4)$ time.
\label{thm:ranking_complexity}
\end{theorem}

\noindent
The remainder of this paper will prove Theorem \ref{thm:ranking_complexity}.
For simplicity, the word $\word{v}$ is assumed to be a necklace representation.
It is well established how to find the lexicographically largest necklace smaller than or equal to some given word.
Such a word can be found in quadratic time using an algorithm form~\cite{Sawada2017}.
Note that the number of necklaces less than or equal to $\word{v}$ corresponds to the number of necklaces less than or equal to the lexicographically largest necklace smaller than $\word{v}$.
From Lemma \ref{lem:ranking_bracelets} it follows that to rank $\word{v}$ with respect to the set of bracelets, it is sufficient to rank $\word{v}$ with respect to the set of necklaces, palindromic necklaces, and enclosing bracelets.
The rank with respect to the set of palindromic necklaces, $RP(\word{v})$ can be computed in $O(k\cdot n^3)$ using the techniques given in Theorem \ref{thm:pallindromic_comlexity} in Section \ref{sec:Ptv}. 
The rank with respect to the set of enclosing bracelets, $RE(\word{v})$ can be computed in $O(k\cdot n^4)$ 
as shown in Theorem \ref{thm:enclosing_complexity} in Section \ref{sec:Stv}. 
As each of these steps can be done independently of each other, the total complexity is $O(k\cdot n^4)$.

This complexity bound is a significant improvement over the naive method of enumerating all bracelets, requiring exponential time in the worst case.
New intuition is provided to rank the palindromic and enclosing cases.
The main source of complexity for the problem of ranking comes from having to consider the lexicographic order of the word under reflection.
New combinatorial results and algorithms are needed to count the bracelets in these cases.

Before showing in detail the algorithmic results that allow bracelets to be efficiently ranked, it is useful to discus the high level ideas.
Lemma \ref{lem:ranking_bracelets} shows our approach to ranking bracelets by dividing the problem into the problems of ranking necklaces, palindromic necklaces and enclosing bracelets.
For both palindromic necklaces and enclosing bracelets, we derive a \emph{canonical form} using the combinatorial properties of these objects.

Using these canonical forms, the number of necklaces smaller than $\word{v}$ is counted in an iterative manner. 
In the palindromic case, this is done by counting the number of necklaces greater than $\word{v}$, and subtracting this from the total number of palindromic necklaces.
In the enclosing case, this is done by directly counting the number of necklaces smaller than $\word{v}$.
For both cases, the counting is done by way of a tree comprised of the set of all prefixes of words of the canonical form.
By partitioning the internal vertices of the trees based on the number of children of the vertices, the number of words of the canonical form may be derived in an efficient manner, forgoing the need to explicitly generate the tree.
This allows the size of these partitions to be computed through a dynamic programming approach.
It follows from these partitions how to count the number of leaf nodes, corresponding to the canonical form.

\begin{theorem}
\label{thm:unranking}
The $z^{th}$ bracelet of length $n$ over $\Sigma$ can be computed in $O(n^5 \cdot k^2 \cdot \log(k))$.
\end{theorem}

\begin{proof}
The unranking process is done through a binary search using the ranking algorithm as a black box.
Let $\word{\alpha}$ be a word which is the bracelet representation of the $z^{th}$ bracelet.
The value of $\word{\alpha}$ is determined iteratively, starting with the first symbol and working forwards.
The first symbol of $\word{\alpha}$ is determined preforming a binary search over $\word{\alpha}$.
For $x \in \word{\alpha}$, the words $x 1^{n - 1}$ and $x k^{n - 1}$ are generated, where $1$ is the smallest symbol in $\Sigma$ and $k$ the largest.
If $RB(x 1^{n - 1}) \leq z \leq RB(x k^{n - 1})$, then the first symbol of $\word{\alpha}$ is $x$, otherwise the new value of $x$ is chosen by standard binary search, being greater than $x$ if $z > RB(x k^{n - 1})$ and less than $x$ if $z < RB(x 1^{n - 1})$.
The $i^{th}$ symbol of $\word{\alpha}$ is done in a similar manner, generating the words $\word{\alpha}_{[1,i-1]} x 1^{n - i - 1}$ and $\word{\alpha}_{[1,i-1]} x k^{n - i - 1}$, converting $\word{\alpha}_{[1,i-1]} x 1^{n - i - 1}$ to a necklace representation using Algorithm 1 due to Sawada and Williams \cite{Sawada2017}.
Repeating this for all $n$ symbols leaves $\word{\alpha}$ as being the bracelet representation of the $z^{th}$ smallest bracelet, i.e. the bracelet with $z - 1$ smaller bracelets.
As the binary search will take $\log(k)$ operations for each of the $n$ symbols, requiring $O(k^2 \cdot n^4)$ time to rank for each symbol at each position.
Therefore the total complexity is $O(n^5 \cdot k^2 \cdot \log(k))$ time.
\end{proof}

\section{Computing the rank $RP(\word{v})$}
\label{sec:Ptv}

To rank palindromic necklaces, it is crucial to analyse 
their combinatorial properties.
This section focuses on providing results on determining unique words representing palindromic necklaces.
We study two cases depending on whether the length $n$ of a palindromic necklace is even or odd.
%
The reason for this division can be seen by considering examples of palindromic necklaces.
If equivalence under the rotation operation is not taken into account, then a word is palindromic if $\word{w} = \word{w}^R$.
If the length $n$ of $\word{w}$ is odd, then if $\word{w} = \word{w}^R$, $\word{w}$ can be written as $\word{\phi} x \word{\phi}^R$, where $\word{\phi} \in \Sigma^{(n - 1)/ 2}$ and $x \in \Sigma$.
For example, the word $aaabaaa$ is equal to $\word{\phi} x \word{\phi}^R$, where $\word{\phi} = aaa$ and $x = b$.
If the length $n$ of $w$ is even, then if $\word{w} = \word{w}^R$, $\word{w}$ can be written as $\word{\psi}\word{\psi}^R$, where $\word{\psi} \in \Sigma^{n/2}$.
For example the word $aabbaa$ is equal to $\word{\psi} \word{\psi}^R$, where $\word{\psi} = aab$.

Once rotations are taken into account, the characterisation of palindromic necklaces becomes more difficult.
It is clear that any necklace $\necklace{a}$ that contains a word of the form $\word{\phi} x \word{\phi}^R$ or $\word{\phi} \word{\phi}^R$ is palindromic.
However this check does not capture every palindromic necklace.
Let us take, for example, the necklace $\necklace{a} = ababab$, which contains two words $ababab$ and $bababa$.
While $ababab$ can neither be written as $\word{\phi} x \word{\phi}^R$ nor $\word{\phi} \word{\phi}^R$, it is still palindromic as $\langle ababab^R\rangle = \langle bababa \rangle = ababab$.
Therefore a more extensive test is required.
As the structure of palindromic words without rotation is different depending on the length being either odd or even, it is reasonable to split the problem of determining the structure of palindromic necklaces into the cases of odd and even length.

The number of palindromic necklaces are counted by computing the number of these characterisations.
This is done by constructing trees containing every prefix of these characterisations.
As each vertex corresponds to the prefix of a word, the leaf nodes of these trees correspond to the words in the characterisations.
By partitioning the tree in an intelligent manner, the number of leaf nodes and therefore number of these characterisations can be computed.
In the odd case this corresponds directly to the number of palindromic necklaces, while in the even case a small transformation of these sets is needed.

\subsection{Odd Length Palindromic Necklaces}
\label{subsec:odd_length_palindromic}

Starting with the odd-length case, {Proposition \ref{prop:odd_phorm} shows that every palindromic necklace of odd length contains {\bf exactly one word} that can be written as $\word{\phi} x \word{\phi}^R$ where $\word{\phi} \in \Sigma^{(n - 1)/2}$ and $x \in \Sigma$.} 
This fact is used to rank the number of bracelets by constructing a tree representing every prefix of a word of the form $\word{\phi} x \word{\phi}^R$ that belongs to a bracelet greater than $\word{v}$.

\begin{proposition}
A necklace $\necklace{w}$ of odd length $n$ is palindromic if and only if there exists exactly one word $\word{u} = \word{\phi} x \word{\phi}^{R}$ such that $\word{v} \in \necklace{w}$, where $\word{\phi} \in \Sigma^{(n - 1)/2}$ and $x \in \Sigma$.
\label{prop:odd_phorm}
\end{proposition}

\begin{proof}
Let $\word{v} \in \necklace{w}$.
If $\word{v}$ is of the form $\word{\phi} x \word{\phi}^R$, then clearly we have that $\word{v} = \word{v}^R$.
In the other direction, for the sake of contradiction assume $\necklace{w}$ is a palindromic necklace of odd length $n$ such that no word $\word{v} \in \necklace{w}$ is of the form $\word{\phi} x \word{\phi}^R$.
Note that the cardinality of $\necklace{w}$ is equal the period of the words in $\necklace{w}$.
As the length of the words in $\necklace{w}$ is odd, so to must be the length of the period.
Given a word $\word{v} \in \necklace{w}$, if $\word{v} \neq \word{v}^R$ then the size of $\necklace{w}$ is equal to $|\necklace{w} \setminus \{\word{v},\word{v}^R\}| + 2$.
As the size of $\necklace{w}$ is odd, there must be at least one word $\word{v} \in \necklace{w}$ where $\word{v} = \word{v}^R$.
For $\word{v} = \word{v}^R$, $\word{v}_1 = \word{v}_n, \word{v}_2 = \word{v}_{n - 1}, \hdots, \word{v}_{\frac{n - 1}{2}} = \word{v}_{\frac{n + 3}{2}}$.
Therefore this word can be expressed as $\word{\phi} x \word{\phi}^R$ where $\word{\phi} = \word{v}_{[1, (n - 1)/2]}$ and $x = \word{v}_{(n + 1)/2}$.

For the remainder of this proof $\word{u}_i$ is used to denote the character at position $(i \bmod n) + 1$ in the word $\word{u}$.
For the sake of contradiction, assume that there exists some pair of words $\word{u},\word{v} \in \necklace{w}$ such that $\word{u} \neq \word{v}$ and both $\word{u} = \word{u}^R$ and $\word{v} = \word{v}^R$.
As both $\word{u}$ and $\word{v}$ belong to the same necklace class, there must exist some rotation $r$ such that  $\langle \word{u} \rangle_{r} = \word{v}$.
Further, as $\word{v} = \word{v}^R$, $\langle \word{u} \rangle_{r} = \word{v}^R$.
Therefore, $\word{u}_{r + i} = \word{v}_i, \word{u}_{n - 1 - r + i} = \word{v}_i$, $\word{u}_{r + i} = \word{v}_{n - i - 1}$, and $\word{u}_{n - 1 - r + i} = \word{v}_{n - i - 1}$.
Further $\word{u}_{i} = \word{u}_{n - i + 1}$ and $\word{v}_{i} = \word{v}_{n - i + 1}$.
Therefore $\word{v}_{i} = \word{u}_{r + i} = \word{u}_{n - r - i + 1} = \word{v}_{2n - 2r - i - 1} = \word{v}_{n - (2n - 2r - i - 1) - 1} = \word{u}_{3r - n + i} = \word{u}_{3r + i}$.
Therefore $\word{u}_{i} = \word{u}_{2r + i}$ implying that $\word{v} = \langle \word{v} \rangle_{2r}$.
Therefore the period of $\word{u}$ must be equal to some common divisor of $2r$ and $n$.
As the length of $n$ is odd, the greatest divisor equals to $GCD(r,n)$.
As such the period must be a factor of $r$, meaning that $\word{u} = \langle\word{u}\rangle_{r} = \word{v}$, contradicting the assumption that $\word{u} \neq \word{v}$.
Therefore there is exactly one word in $\necklace{w}$ of the form $\word{\phi} x \word{\phi}^R$.
\end{proof}

\begin{corollary}
The number of palindromic necklaces of odd length $n$ over $\Sigma$ equals $k^{(n + 1)/2}$.
\label{prop:num_odd_necklaces}
\end{corollary}

\begin{proof}
It follows from Proposition \ref{prop:odd_phorm} that for every palindromic necklace $\necklace{w}$ of length $n$, there exists exactly one word $\word{\phi} \in \Sigma^*$ and symbol $x \in \Sigma$ such that $\word{\phi} x \word{\phi}^R$.
Hence, the number of palindromic necklaces equals the number of words of the form $\word{\phi} x \word{\phi}^R$ with length $n$.
Note that for the length of $\word{\phi} x \word{\phi}^R$ to be $n$, the length of $\word{\phi}$ must be $\frac{n - 1}{2}$.
Therefore the number of values of $\word{\phi}$ is $k^{(n - 1)/2}$.
As there are $k$ values of $x$, the number of values of $\word{\phi} x \word{\phi}^R$ is $k^{(n + 1)/2}$.
\end{proof}

\noindent
The problem now becomes to rank a word $\word{v}$ with respect to the odd length palindromic necklaces utilising their combinatorial properties.
Let $\word{v} \in \Sigma^n$ be a word of odd length $n$.
We define the set $\mathcal{PO}(\word{v})$, where $\mathcal{PO}$ stands for palindromic odd length.
The set $\mathcal{PO}(\word{v})$ contains one word representing each palindromic bracelet of odd length $n$ that is greater than $\word{v}$.
\begin{align*}
\mathcal{PO}(\word{v}) := \Big\{ \word{w} \in \Sigma^n: \word{w} = \word{\phi} x \word{\phi}^R, ~\text{where}~ \langle \word{w} \rangle > \word{v}, ~ \word{\phi} \in \Sigma^{(n - 1)/2}, ~x \in \Sigma\Big\}.
\end{align*}
\noindent
As each word will correspond to a unique palindromic necklace of length $n$ greater than $\word{v}$, and every palindromic necklace greater than $\word{v}$ will correspond to a word in $\mathcal{PO}(\word{v})$, the number of palindromic necklaces greater than $\word{v}$ is equal to $|\mathcal{PO}(\word{v})|$.
Using this set the number of necklaces less than $\word{v}$ can be counted by subtracting the size of $\mathcal{PO}(\word{v})$ from the total number of odd length palindromic necklaces, equal to $k^{(n + 1)/2}$ (Corollary~\ref{prop:num_odd_necklaces}).

\begin{figure}
    \centering
    \begin{tabular}{l l}
        \includegraphics[scale=0.7]{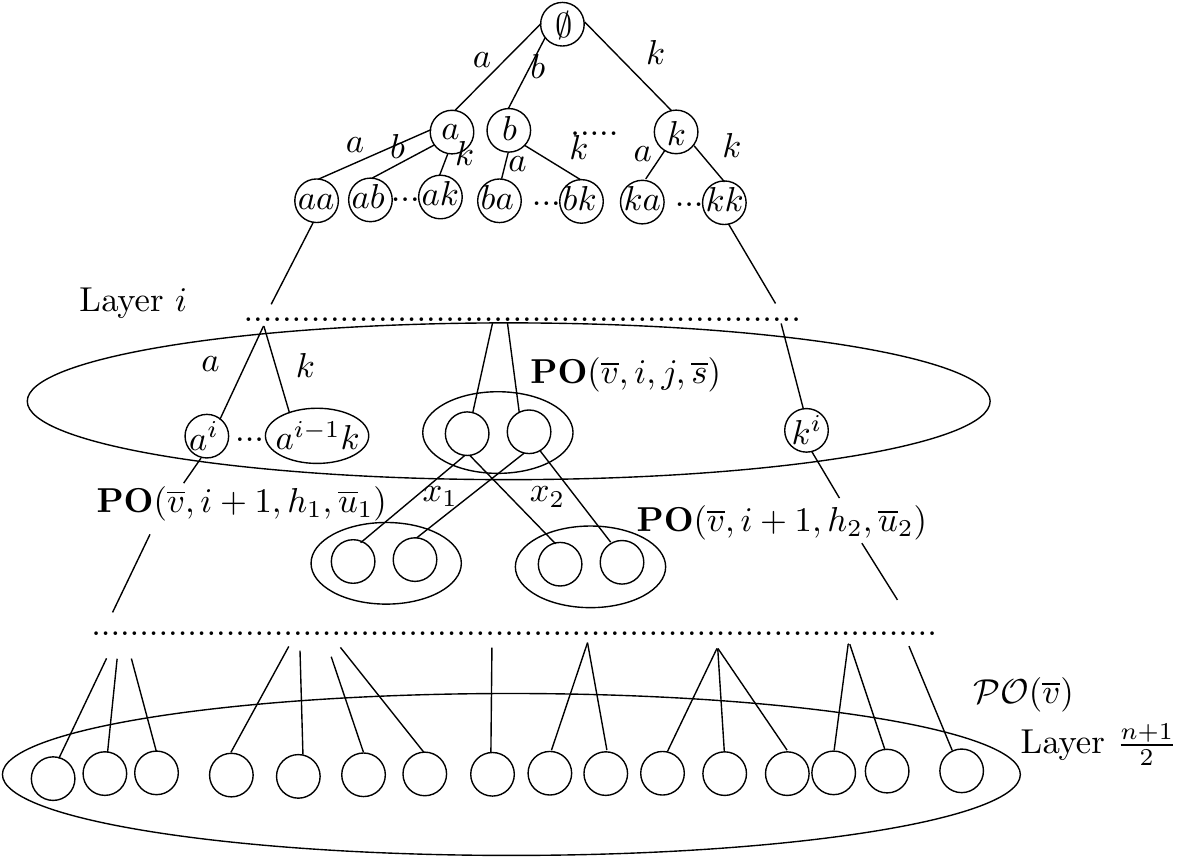} & \includegraphics[scale=0.4]{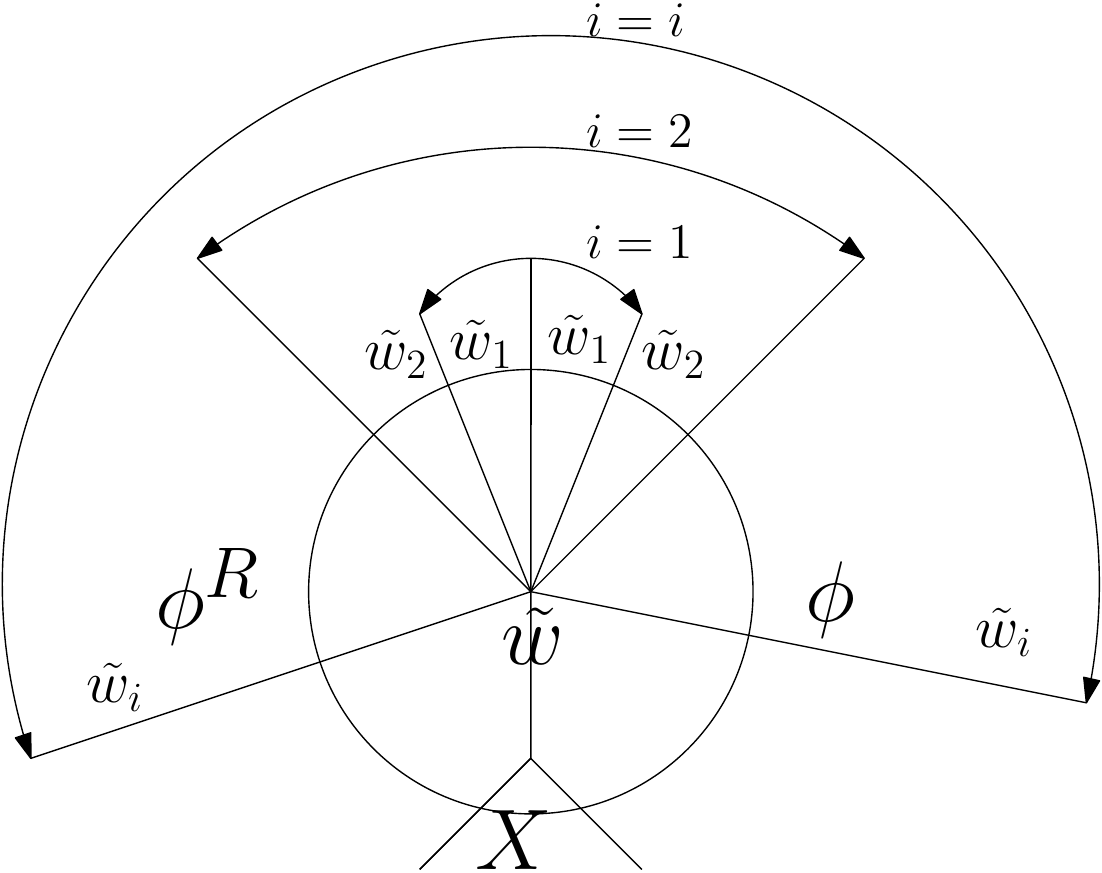}
    \end{tabular}
    \caption{(Left) The relationship between $\mathbf{PO}(\word{v},i,j,\word{s})$ with the tree $\mathcal{TO}(\word{v})$ and $\mathbf{PO}(\word{v})$. (right) Example of the order for which characters are assigned. Note that at each step the choices for the symbol $\necklace{w}_i$ is constrained in the no subword of $\necklace{w}_{[1,i]} \necklace{w}_{[1,i]}^R$ is greater than or equal to $\word{v}$.
    }
    \label{fig:PO_relationship}
\end{figure}

\noindent
\textbf{High level idea for the Odd Case.}
Here we provide a high level idea for the approach we follow for computing $\mathcal{PO}(\word{v})$.
Let $\word{v}$ have a length $n$. 
Since $\mathcal{PO}(\word{v})$ only contains words of the form $\word{\phi} x \word{\phi}^R$, where $\word{\phi} \in \Sigma^{(n - 1)/2}$ and $x \in \Sigma$, 
we have that $\word{w}_i = \word{w}_{n - i}$ for every $\word{w} \in \mathcal{PO}(\word{v})$.As the lexicographically smallest rotation of every $\word{w} \in \mathcal{PO}(\word{v})$ must be greater than $\word{v}$, it follows that any word rotation of $\word{w}$ must be greater than $\word{v}$ and therefore every subword of $\word{w}$ must also be greater than or equal to the prefix of $\word{v}$ of the same length.
This property is used to compute the size of $\mathcal{PO}(\word{v})$ by iteratively considering the set of prefixes of each word in $\mathcal{PO}(\word{v})$ in increasing length representing them with the tree $\mathcal{TO}(\word{v})$.
As generating $\mathcal{TO}(\word{v})$ directly would require an exponential number of operations, a more sophisticated approach is needed for the calculation of $|\mathcal{PO}(\word{v})|$ based on partial information.

As the tree $\mathcal{TO}(\word{v})$ is a tree of prefixes, vertices in $\mathcal{TO}(\word{v})$ are referred to by the prefix they represent.
So $\word{u} \in \mathcal{TO}(\word{v})$ refers to the unique vertex in $\mathcal{TO}(\word{v})$ representing $\word{u}$.
The root vertex of $\mathcal{TO}(\word{v})$ corresponds to the empty word.
Every other vertex $\word{u} \in \mathcal{TO}(\word{v})$ corresponds to a word of length $i$, where $i$ is the distance between $\word{u}$ and the root vertex.
Given two vertices $\word{p},\word{c} \in \mathcal{TO}(\word{v})$, $\word{p}$ is the parent vertex of a child vertex $\word{c}$ if and only if $\word{c} = \word{p} x$ for some symbol $x \in \Sigma$.
The $i^{th}$ layer of $\mathcal{TO}(\word{v})$ refers to all representing words of length $i$ in $\mathcal{TO}(\word{v})$ .
The size of $\mathcal{PO}(\word{v})$ is equivalent to the number of unique prefixes of length $\frac{n + 1}{2}$ of words of the palindromic form $\word{\phi} x \word{\phi}^R $ in $\mathcal{PO}(\word{v})$.
This set of prefixes corresponds to the vertices in the layer $\frac{n + 1}{2}$ of $\mathcal{TO}(\word{v})$.
Therefore the maximum depth of $\mathcal{TO}(\word{v})$ is $\frac{n + 1}{2}$.

To speed up computation, each layer of $\mathcal{TO}(\word{v})$ is partitioned into sets that allow the size of $\mathcal{PO}(\word{v})$ to be efficiently computed.
This partition is chosen such that the size of the sets in layer $i + 1$ can be easily derived from the size of the sets in layer $i$.
As these sets are tied to the tree structure, the obvious property to use is the number of children each vertex has.
As each vertex $\word{u} \in \mathcal{TO}(\word{v})$ represents a prefix of some word $\word{w} \in \mathcal{PO}(\word{v})$, the number of children of $\word{u}$ is the number of symbols $x \in \Sigma$ such that $\word{u} x$ is a prefix of some word in $\mathcal{PO}(\word{v})$.
Recall that every word in $\word{w} \in \mathcal{PO}(\word{v})$ has the form $\word{\phi} x \word{\phi}^R$, and that there is no subword of $\word{w}$ that is less than $\word{v}$.
Therefore if $\word{u} \in \mathcal{TO}(\word{v})$, there must be no subword of $\word{u}^R \word{u}$ that is less than $\word{v}$.
Hence the number of children of $\word{u}$ is the number of symbols $x \in \Sigma$ such that no subword of $x \word{u}^R \word{u} x$ is less than the prefix of $\word{v}$ of the same length.
As $\word{u}^R \word{u}$ has no subword less than $\word{v}$, $x \word{u}^R \word{u} x$ will only have a subword that is less than $\word{v}$ if either (1) $x \word{u}^R \word{u} x < \word{v}$ or (2) there exists some suffix of length $j$ such that $(\word{u}^R \word{u})_{[2i - j,2i]} = \word{v}_{[1,j]}$ and $x < \word{v}_{j + 1}$.
For the first condition, let $\word{s} \sqsubseteq_{2i} \word{v}$.
By the definition of strictly bounding subwords (Definition \ref{def:bounding}), $x \word{u}^R \word{u} x < \word{v}$ if and only if $x \word{s} x < \word{v}$.
Note that this ignores any word $\word{u}$ where $\word{u}^R \word{u} \sqsubseteq \word{v}$.
The restriction to strictly bounded words is to avoid the added complexity caused by Proposition \ref{prop:complexity_XW}, where the word that bounds $x \word{s} x$ might not be the word that bounds $x \word{u}^R \word{u} x$.
For the second property, let $j$ be the length of the longest suffix of $\word{u}^R \word{u}$ that is a prefix of $\word{v}$.
From Lemma 1 due to Sawada and Williams \cite{Sawada2017}, there is some suffix of $\word{u}^R \word{u} x$ that is smaller than $\word{v}$ if and only if $x < \word{v}_{j + 1}$.
The $i^{th}$ layer of $\mathcal{TO}(\word{v})$ is partitioned into $n^2$ sets  $\mathbf{PO}(\word{v},i,j,\word{s})$, for every $i \in [\frac{n + 1}{2}], j \in [2i]$ and $\word{s} \sqsubseteq_{2i} \word{v}$.

\begin{definition}
Let $i \in [\frac{n + 1}{2}], j \in [2i]$ and $\word{s} \sqsubseteq_{2i} \word{v}$.
The set $\mathbf{PO}(\word{v},i,j,\word{s})$ contains every prefix $\word{u} \in \mathcal{TO}(\word{v})$ of length $i$ where (1) the longest suffix of $\word{u}_{[1,i]}^R \word{u}_{[1,i]}$  which is a prefix of $\word{v}$ has a length of $j$ and (2) The word $\word{u}_{[1,i]}^R \word{u}_{[1,i]}$ is strictly bounded by $\word{s}$.
\end{definition}

An overview of the properties used by $\mathbf{PO}(\word{v},i,j,\word{s})$ is given in Figures \ref{fig:PO_relationship} and \ref{fig:word_properties_visual}.
It follows from the earlier observations that each vertex in $\mathbf{PO}(\word{v},i,j,\word{s})$ has the same number of children.
Lemma \ref{lem:PO_cartesian} strengthens this observation, showing that given $\word{a},\word{b} \in \mathbf{PO}(\word{v},i,j,\word{s})$, $\word{a} x \in \mathbf{PO}(\word{v},i + 1,j',\word{s}')$ if and only if $\word{b} x \in \mathbf{PO}(\word{v}, i + 1, j', \word{s}')$.

\begin{figure}
    \centering
    \includegraphics[scale=0.7]{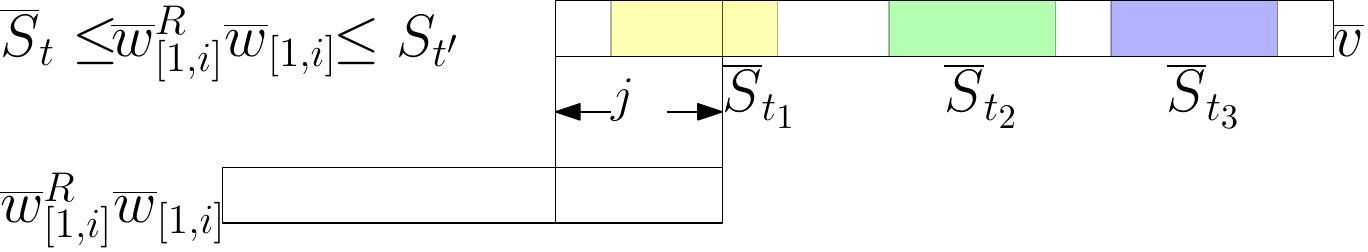}
    \caption{Visual representation of the properties of $\word{w}^R_{[1,i]} \word{w}_{[1,i]} \in \mathbf{PO}(\word{v},i,j,\word{s})$.}
    \label{fig:word_properties_visual}
\end{figure}

The remainder of this section establishes how to count the size of $\mathbf{PO}(\word{v},i,j,\word{s})$ and the number of children vertices for each vertex in $\mathbf{PO}(\word{v},i,j,\word{s})$.
The first step is to formally prove that all vertices in $\mathbf{PO}(\word{v},i,j,\word{s})$ have the same number of children vertices.
This is shown in Lemma \ref{lem:PO_cartesian} by proving that given two vertices $\word{a},\word{b} \in \mathbf{PO}(\word{v},i,j,\word{s})$, if the vertex $\word{a}'= \word{a} x$ for $x \in \Sigma$ belongs to the set $\mathbf{PO}(\word{v},i + 1, j', \word{s}')$, so to does $\word{b}' = \word{b} x$.

\begin{lemma}
Let $\word{a},\word{b} \in  \mathbf{PO}(\word{v},i,j,\word{s})$ and let $x \in \Sigma$.
If the vertex $\word{a}' = \word{a} x$ belongs to $\mathbf{PO}(\word{v}, i + 1, j', \word{s}')$, the vertex $\word{b}' = \word{b} x$ also belongs to $\mathbf{PO}(\word{v}, i + 1, j', \word{s}')$.
Furthermore the value of $j'$ and $\word{s}'$ can be computed in constant time from the values of $j,\word{s}$ and $x$.
\label{lem:PO_cartesian}
\end{lemma}

\begin{proof}
By the definition of the set $\mathbf{PO}(\word{v},i,j,\word{s})$, the last $j$ symbols of $\word{a}^R \word{a}$ and $\word{b}^R \word{b}$ are equal to $\word{v}_{[1,j']}$.
Therefore if $j' > 0$, $x$ must be equal to $\word{v}_{j + 1}$, satisfying this observation.
On the other hand, if $j' = 0$ then $x$ must be greater than $\word{v}_{j + 1}$.
Following Lemmas \ref{lem:bounding_same_wx} and \ref{lem:bounding_same_xw}, if $\word{s}'$ bounds $x \word{a}^R \word{a} x$ and $\word{s}$ bounds both $\word{a}^R \word{a}$ and $\word{b}$, then $\word{s}'$ also bounds $x \word{b}^R \word{b} x$.
Hence $\word{b}'$ must also belong to $\mathbf{PO}(\word{v}, i + 1, j', \word{s}')$.

To compute the value of $j'$ and $\word{s}'$ in constant time, assume that the arrays $XW$ and $WX$ as defined in Section \ref{subsec:bounding_subwords}.
Note that if $x < \word{v}_{j + 1}$, there is no such value of $j'$ or $\word{s}'$ as the suffix of $x \word{a}^R \word{a} x$ of length $j + 1$ is smaller than $\word{v}$, contradicting the definition of the set.
If $x = \word{v}_{j + 1}$ then the value of $j'$ must be $j + 1$.
Otherwise, the value of $j'$ is $0$ following Lemma 1 of Sawada and Williams \cite{Sawada2017}.
The value $\word{s}'$ can be derived using $WX$ and $XW$ by finding the word $\word{u} = WX[\word{s},x]$ that bounds $\word{s} x$, then $\word{s}' = XW[\word{u}, x]$ that bounds $x \word{u}$.
Therefore the value of $j'$ and $\word{s}'$ can be computed in constant time.
\end{proof}

\noindent
\textbf{Computing the size of $\mathbf{PO}(\word{v},i,j',\word{s}')$.}
Lemma \ref{lem:PO_cartesian}, provides enough information to compute the size of $\mathbf{PO}(\word{v},i,j',\word{s}')$ once the size of $\mathbf{PO}(\word{v},i - 1, j,\word{s})$ has been computed for each value of $j \in [2(i - 1)]$ and $\word{s} \in \mathbf{S}(\word{v},2(i - 1))$.
At a high level, the idea is to create an array, $SizePO$, storing the size of the $\mathbf{PO}(\word{v},i,j',\word{s}')$  for every value of $i \in [\frac{n - 1}{2}], j \in [2i]$ and $\word{s} \sqsubseteq_{2i} \word{v}$.
For simplicity, let the value of $SizePO[i,j,\word{s}]$ be the size of $|\mathbf{PO}(\word{v},i,j,\word{s})|$.

Lemma \ref{lem:SizeOfPO(i,j,s)} formally provides the method of computing $SizePO[i,j,\word{s}]$ for every  $j \in [2i]$ and $\word{s} \sqsubseteq_{2i} \word{v}$ once $SizePO[i - 1,j',\word{s}']$ has been computed for every $j' \in [2i - 2]$ and $\word{s} \sqsubseteq_{2i - 2} \word{v}$.
Observe that each vertex $a \in \mathbf{PO}(\word{v},i,j',\word{s}')$ represents a prefix $\word{a}' x$ where $\word{a}'$ is either in $\mathbf{PO}(\word{v},i - 1, j,\word{s})$, for some value of $j$ and $\word{s}$, or $\word{a}' \sqsubseteq \word{v}$.
Using this, the high level idea is to derive the values of $j'$ and $\word{s}'$ for each $j \in [2(i - 1)], \word{s} \in \mathbf{S}(\word{v},2(i -1 ))$ and $x \in \Sigma$.
Once the values $j'$ and $\word{s}'$ have been derived, the value of $SizePO[i,j',\word{s}']$ is increased by the size of $\mathbf{PO}(\word{v},i-1,j,\word{s})$.
Repeating this for every value of $j,\word{s}$ and $x$ will leave the value of $SizePO[i,j',\word{s}']$ as the number of vertices in $\mathbf{PO}(\word{v},i,j',\word{s}')$ representing words of the form $\word{a} x$ where $\word{a} \not\sqsubseteq \word{v}$.
As each set $\mathbf{PO}(\word{v},i,j,\word{s})$ may have children in at most $k$ sets $\mathbf{PO}(\word{v},i + 1,j',\word{s}')$, the number of vertices in $\mathbf{PO}(\word{v},i + 1,j',\word{s}')$ with a parent vertex in $\mathbf{PO}(\word{v},i,j,\word{s})$ can be computed in $O(k \cdot n^2)$ by looking at every argument of $j \in [2i]$ and $\word{s} \sqsubseteq_{2i} \word{v}$.

To account for the vertices in $\mathbf{PO}(\word{v},i,j',\word{s}')$ of the form  $\word{b} x$ where $\word{b}^R \word{b} \sqsubseteq \word{v}$, a similar process is applied to each pair $\word{s} \in \mathbf{S}(\word{v},2(i - 1))$ and $x \in \Sigma$.
For each pair, the values $\word{s}'$ and $j'$ are derived in the same manner as Lemma \ref{lem:PO_cartesian} utilising the tables $XW$ and $WX$.
Once derived, the value of $SizePO[i,j',\word{s}']$ is increased by one, to account for the vertex $\word{s} x$.
As the values of $j'$ and $\word{s}'$ can be computed in $O(n)$ time from the value of $x$ and $\word{s}$, the number of vertices in $\mathbf{PO}(\word{v},i + 1,j',\word{s}')$ where the parent vertex is a subword of $\word{v}$ can be computed in $O(k \cdot n^2)$ time.

\begin{lemma}
\label{lem:SizeOfPO(i,j,s)}
Given the size of $\mathbf{PO}(\word{v},i,j,\word{s})$ for $i \in \left[\frac{n - 3}{2}\right]$ and every $j \in [2i],\word{s} \sqsubseteq_{2i} \word{v}$, the size of $\mathbf{PO}(\word{v},i + 1,j',\word{s}')$ for every $j' \in [2i + 2],\word{s}' \sqsubseteq_{2i + 2} \word{v}$ can be computed in $O(k \cdot n^2)$ time.
\end{lemma}

\begin{proof}
Assume that $WX$ and $XW$ have been precomputed.
Further assume that the array $SizePO$ has be initialised such that $SizePO[i,j,\word{s}] = |\mathbf{PO}(\word{v},i,j,\word{s})|$ for every value of $j \in [2i]$ and $\word{s} \in \mathbf{S}(\word{v},2i)$, and $SizePO[i + 1,j,\word{s}] = 0$ for every $j' \in [2i + 2]$, and $\word{s}' \in \mathbf{S}(\word{v},2i)$.

The first step is to count the number of vertices in $\mathbf{PO}(\word{v},i + 1,j',\word{s}')$ representing words of the form $\word{a} x$ where $\word{a} \not\sqsubseteq \word{v}$.
This is done by checking each $j \in [2i], \word{s} \in \mathbf{S}(\word{v}, 2i)$, and $x \in \Sigma$.
For each $j, \word{s}$ and $x$, the values $j'$ and $\word{s}'$ are derived in constant time as in Lemma \ref{lem:PO_cartesian}.
Following Lemma \ref{lem:PO_cartesian}, every vertex $\word{a} \in \mathbf{PO}(\word{v},i,j,\word{s})$ will have some child vertex in $\word{a}' \in \mathbf{PO}(\word{v},i + 1,j',\word{s}')$ such that the last symbol of the word $\word{a}'$ is equal to $x$.
Therefore the value of $SizePO[i + 1, j', \word{s}']$ is increased by the value of $SizePO[i,j,\word{s}]$.
Repeating this for every value of $j,\word{s}$ and $x$ will leave the value of $SizePO[i + 1, j', \word{s}']$ equal to the number of vertices in $\mathbf{PO}(\word{v},i + 1,j',\word{s}')$ of the form $\word{a} x$ where $\word{a} \not\sqsubseteq \word{v}$.
As there are $n$ possible value of both $j$ and $\word{s}$, and $k$ values of $x$, this process will take $O(n^2 \cdot k)$ operations.

To compute the number vertices in $\mathbf{PO}(\word{v},i + 1,j',\word{s}')$ of the form  $\word{b} x$ where $\word{b} \sqsubseteq \word{v}$, a similar process is applied to each pair $\word{s} \in \mathbf{S}(\word{v},2i)$ and $x \in \Sigma$.
Formally, for each pair of $\word{s}$ and $x$, the first step is to check that $\word{s} = \word{s}^R$.
This can be done in linear time by comparing the two strings.
This check ensures that new word will be palindromic.
The second check is that $x \word{s} x \not \subseteq \word{v}$.
This is to ensure that the new word is not counted in the next layer.
This can be done by finding the word $\word{s}'$ in the same manner as in Lemma \ref{lem:PO_cartesian}, and checking if the word $\word{u}'$ preceding $\word{s}'$ in the ordered set $\mathbf{S}(\word{s},2i + 2)$ is equal to $x \word{s} x$.
Let $j$ be the length of the longest suffix of $\word{s}$ that is a prefix of $\word{v}$.
The value of $j$ can be found in linear time by using a simple pattern matching algorithm on $\word{s}$ and recording the final state.
The value of $j'$ can be found form the value of $j$ and $x$ using Lemma \ref{lem:PO_cartesian} in constant time.
Once $j'$ and $\word{s}'$ have been derived, the value of $SizePO[i + 1, j', \word{s}']$ can be increased by 1.
As there are $n$ possible values of $\word{s}, k$ possible values of $x$, and at most $O(n)$ operations are required for each pair, this process will take $O(n^2 \cdot k)$ operations.
Therefore the total complexity is $O(n^2 \cdot k)$.
\end{proof}

Once the size of $\mathbf{PO}(\word{v},i,j,\word{s})$ has been computed for every $i \in [\frac{n-1}{2}],j\in[2i],\word{s}\in\mathbf{S}(\word{v},2i)$, the final step is to compute $|\mathcal{PO}(\word{v})|$.
The high level idea is to determine the number of vertices in $\mathcal{PO}(\word{v})$ are children of a vertex in $\mathbf{PO}(\word{v},\frac{n - 1}{2},j,\word{s})$.
The set $\mathbf{X}(\word{v},j,\word{s}) \subseteq \Sigma$ is introduced to help with this goal.
Let $\mathbf{X}(\word{v},j,\word{s})$ contain every symbol $x \in \Sigma$ such that $\word{a} x \word{a}^R \in \mathcal{PO}(v)$ where $\word{a} \in \mathbf{PO}(\word{v},\frac{n - 1}{2},j,\word{s})$.
By the definition of $\mathbf{X}(\word{v},j,\word{s})$, $|\mathbf{X}(\word{v},j,\word{s})| \cdot |\mathbf{PO}(\word{v},\frac{n - 1}{2},j,\word{s})|$ equals the number of words $\word{w} \in \mathcal{PO}(\word{v})$ where $(\word{w}_1 \hdots \word{w}_{(n - 1)/2}) \in \mathbf{PO}(\word{v},i,j,\word{s})$.
Lemma \ref{lem:Z(v,i,j)Computaiton} shows how to compute the size of $\mathbf{X}(\word{v},j,\word{s})$ in $O(k \cdot n)$ time.

\begin{lemma}
\label{lem:Z(v,i,j)Computaiton}
Let $\mathbf{X}(\word{v},j,\word{s})$ contain every symbol in $\Sigma$ such that $\word{a} x \word{a}^R \in \mathcal{PO}(v)$ where $\word{a} \in \mathbf{PO}(\word{v},\frac{n - 1}{2},j,\word{s})$.
The size of $\mathbf{X}(\word{v},j,\word{s})$ can be computed in $O(k \cdot n)$ time.
\end{lemma}

\begin{proof}
The size of $\mathbf{X}(\word{v},j,\word{s})$ can be computed in a direct manner by checking if $x \in \mathbf{X}(\word{v},j,\word{s})$ for each $x \in \Sigma$.
Given some $x \in \Sigma$, note that if $x < \word{v}_{j + 1}$ then there exists some rotation of $\langle \word{a} x \word{a}^R \rangle$ that is smaller than $\word{v}$.
Let $x \geq \word{v}_{j + 1}$.
For $x$ to be a member of $\mathbf{X}(\word{v},j,\word{s})$ observe that for $\langle \word{a} x \word{a}^R \rangle$ to be greater than $\word{v}$, $\word{v}_{[1,j]} x \word{a}^R \word{a}$ must be greater than $v$.
Using the bound given by $\word{s}$ gives $\langle\word{a} x \word{a}^R \rangle > \word{v}_{[1,j]} x \word{s}$.
Therefore if $\word{v}_{[1,j]} x \word{s} \geq \word{v}$, $x \in \mathbf{X}(\word{v},j,\word{s})$.
In the other hand, if $\word{v}_{[1,j]} x \word{s} < \word{v}$, then note that $\word{s} < \word{v}_{[j + 2, n + j]}$.
Therefore $\word{a}^R \word{a} < \word{v}_{[j + 2, n + j]}$ as it is bounded by $\word{s}$.
Hence $\langle \word{a} x \word{a}^R \rangle < \word{v}$.
Therefore, $x \in \mathbf{X}(\word{v},j,\word{s})$ if and only if $\word{v}_{[1,j]}  x \word{s} \geq \word{v}$.
As this can be checked in $O(n)$ steps by directly comparing the two words, and there are $k$ values of $z$ to check, the total complexity is $O(k \cdot n)$.
\end{proof}

\noindent
\textbf{Converting $SizePO$ to $|\mathcal{PO}(\word{v})|$.}
The final step in computing $\mathcal{PO}(\word{v})$ is to convert the cardinality of $\mathbf{PO}(\word{v},i,j,\word{s})$ to the size of $\mathcal{PO}(\word{v})$.
Lemma \ref{lem:SizeOfPO} provides a formula for counting the size of $\mathcal{PO}(\word{v})$.
Combining this formula with the techniques given in Lemma \ref{lem:SizeOfPO(i,j,s)} an algorithm for computing the size of $\mathcal{PO}(\word{v})$ directly follows.

It follows from Lemma \ref{lem:PO_cartesian} that the number of words in $\mathcal{PO}(\word{v})$ with a prefix in $\mathbf{PO}\left(\word{v},\frac{n - 1}{2},j,\word{s}\right)$ is equal to the cardinality of $\mathbf{PO}\left(\word{v},\frac{n - 1}{2},j,\word{s}\right)$ multiplied by the size of $\mathbf{X}(\word{v},j,\word{s})$.
Similarly the number of words in $\mathcal{PO}(\word{v})$ with a prefix $\word{u}$ of length $\frac{n - 1}{2}$ where $\word{u}^R \word{u} \sqsubseteq \word{v}$ can be determined using $\mathbf{X}(\word{v},j,\word{u}^R \word{u})$.
The main difference in this case is that if $\word{u}^R\word{u} = \word{v}_{[j + 2,n+j]}$, where $j$ is the length of the longest suffix of $\word{u}^R\word{u}$ that is a prefix of $\word{v}$, then the number of words in $\mathcal{PO}(\word{v})$ where  $\word{u}$ is a prefix is 1 fewer than for the number of words strictly bounded by $\word{u}^R \word{u}$, i.e. $|\mathbf{X}(\word{v},J(\word{s},\word{v}),\word{s})| - 1$.
Lemma \ref{lem:SizeOfPO} provides the procedure to compute $|\mathcal{PO}(\word{v})|$.

\begin{lemma}
Let $J(\word{s},\word{v})$ return the length of the longest suffix of $\word{s}$ that is a prefix of $\word{v}$.
The size of $\mathcal{PO}(\word{v})$ is equal to \\
$\sum\limits_{\word{s} \in \mathbf{S}(\word{v},n - 1)} \left(\sum\limits_{j = 1}^{n - 1} |\mathbf{X}(\word{v},j,\word{s})| \cdot |\mathbf{PO}\left(\word{v},\frac{n - 1}{2},j,\word{s}\right)|\right) + \begin{cases}
0 & \word{s} \neq \phi \phi^R\\
|\mathbf{X}(\word{v},J(\word{s},\word{v}),\word{s})| &\word{s} \neq \word{v}_{[j + 2, n + j]}\\
|\mathbf{X}(\word{v},J(\word{s},\word{v}),\word{s})| - 1 &\word{s} = \word{v}_{[j + 2, n + j]}
\end{cases}$\\
Further this can be computed in $O(k \cdot n^3 \cdot \log(n))$ time.
\label{lem:SizeOfPO}
\end{lemma}

\begin{proof}
From Lemma \ref{lem:Z(v,i,j)Computaiton} the size of the set $\mathbf{X}(\word{v},j,\word{s})$ can be computed in $O(n \cdot k)$ operations.
By the definition of $\mathbf{X}(\word{v},j,\word{s})$, $|\mathbf{X}(\word{v},j,\word{s})| \cdot |\mathbf{PO}(\word{v},\frac{n - 1}{2},j,\word{s})|$ is the number of words $\word{w} \in \mathcal{PO}(\word{v})$ where $\word{w}_{[1,(n - 1)/2]} \in \mathbf{PO}(\word{v},\frac{n - 1}{2},j,\word{s})$.
Therefore $\sum\limits_{\word{s} \in \mathbf{S}(\word{v},n - 1)} \left(\sum\limits_{j = 1}^{n - 1} |\mathbf{X}(\word{v},j,\word{s})| \cdot |\mathbf{PO}(\word{v},\frac{n - 1}{2},j,\word{s})|\right)$ will count every word $\word{w} \in \mathcal{PO}(\word{v})$  where $\word{w}_{[1,(n - 1)/2]} \in \mathbf{PO}(\word{v},\frac{n - 1}{2},j,\word{s})$ for some arguments $j \in [|\mathbf{v}| - 1], \word{s} \in \mathbf{S}(\word{v}, n - 1)$.
As there are $n^2$ possible values of $j$ and $\word{s}$, and computing $|\mathbf{X}(\word{v},j,\word{s})| \cdot |\mathbf{PO}(\word{v},\frac{n - 1}{2},j,\word{s})|$ requires $O(k \cdot n)$ steps, the total complexity of counting $\sum\limits_{\word{s} \sqsubseteq \word{v}} \left(\sum\limits_{j = 1}^{n - 1} |\mathbf{X}(\word{v},j,\word{s})| \cdot |\mathbf{PO}(\word{v},\frac{n - 1}{2},j,\word{s})|\right)$ is $O(n^3 \cdot k)$.

For words of the form $\word{\phi}^R x \word{\phi}$ where $\word{\phi}^R\word{\phi} \sqsubseteq \word{v}$ note that for every character in $\mathbf{X}(\word{v},j,\word{s})$, $\langle x \word{\phi} \word{\phi}^R \rangle \geq \word{v}$.
Further, as $\langle \word{\phi}^R \word{\phi} x \rangle = \word{v}$ only when $\word{\phi}^R\word{\phi} = \word{v}_{[j + 2, |v| + j]}$ and $x = \word{v}_{j + 1}$, the number of words of this form is $|\mathbf{X}(\word{v},j,\word{s})|$, when $\word{\phi}^R\word{\phi} \neq \word{v}_{[j + 2, |v| + j]}$, and $|\mathbf{X}(\word{v},j,\word{s})| - 1$ otherwise.
As the conditions can be checked in $O(n)$ time, $\mathbf{X}(\word{v},j,\word{s})$ can be computed in $O(n \cdot k)$ time, and there are $O(n)$ subwords in $\mathbf{S}(\word{v},n - 1)$, the total complexity of computing $\sum\limits_{\word{s} \in \mathbf{S}(\word{v},n)} \begin{cases}
|\mathbf{X}(\word{v},J(\word{s},\word{v}),\word{s})| & \word{s} = \phi \phi^R \text{ and } \word{s} \neq \word{v}_{[j + 2, n + j]}\\
|\mathbf{X}(\word{v},J(\word{s},\word{v}),\word{s})| - 1 & \word{s} = \phi \phi^R \text{ and } \word{s} = \word{v}_{[j + 2, n + j]}\\
0 & \word{s} \neq \phi \phi^R
\end{cases}$ is $O(n^2 \cdot k)$.
Therefore the total complexity of computing the size of $\mathcal{PO}(\word{v})$ from the array $SizePO[i,j,\word{s}]$ is $O(n^3 \cdot k)$.
In order to compute the array $SizePO$ a total of $O(k \cdot n^3 \cdot \log(n))$ operations are needed.
Hence the total complexity is $O(k \cdot n^3 \cdot \log(n))$.
\end{proof}

\subsection{Even Length Palindromic Necklaces}
\label{subsec:even_length_palindromic}

Section \ref{subsec:odd_length_palindromic} shows how to rank $\word{v}$ within the set of odd length palindromic necklaces.
This leaves the problem of counting even length palindromic necklaces.
As in the odd case, the first step is to determine how to characterise these words.
Proposition \ref{prop:even_phorm} shows that every palindromic necklace will have at least one word of either the form $\word{\phi} \word{\phi}^R$, where $\word{\phi} \in \Sigma^{n/2}$, or $x \word{\phi} y \word{\phi}^R$, where $x,y \in \Sigma$ and $\word{\phi} \in \Sigma^{(n/2) - 1}$.
Proposition \ref{prop:even_phorm} is strengthened by Propositions \ref{prop:shared_even_form} and \ref{prop:even_one_phrom}, showing that each palindromic necklace of even length will have no more than two words of either form.
Lemmas \ref{lem:PE_Properties}, \ref{lem:PE}, \ref{lem:PS} and \ref{lem:complexity_even_overall} use these results a similar manner to Section \ref{subsec:odd_length_palindromic} to count the number of palindromic necklaces of even length.

\begin{proposition}
\label{prop:even_phorm}
A necklace $\necklace{w}$ of even length $n$ is palindromic if and only if there exists some word $\word{u} \in \necklace{w}$ where either (1) $\word{u} = x \word{\phi} y \word{\phi}^R$ where $x,y \in \Sigma$ and $\word{\phi} \in \Sigma^{(n/2) - 1}$, or (2) $\word{u} = \word{\phi} \word{\phi}^R$ where $\word{\phi} \in \Sigma^{n/2}$.
\end{proposition}

\begin{proof}
Given a word $\word{u}$ of the form $x \word{\phi} y \word{\phi}^R$ where $\word{u} \in \necklace{w}$, $\word{u}^R$ is equal to $\word{\phi} y \word{\phi}^R x$.
Following this observation $\word{u} = \langle \word{u}^R\rangle_1$.
Therefore for every word in $\necklace{w}$ the reflection is also in $\necklace{w}$.
Similarly, given a word $\word{u} \in \necklace{w}$ of the form $\word{\phi} \word{\phi}^R$, $\word{u} = \word{i}^R$, therefore for every word in the necklace $\necklace{w}$, the reflection is also in $\necklace{w}$.

In the other direction, let $\necklace{w}$ be a palindromic necklace of even length $n$.
If there is any word $\word{u} \in \necklace{w}$ such that $\word{u} = \word{u}^R$, then the word must be of the form $\word{\phi} \word{\phi}^R$.
Therefore for the sake of contraction, assume every word $\word{u} \in \necklace{w}$ must not be equal to $\word{u}^R$.
As $\necklace{w}$ is palindromic, there exists some rotation $i$ such that $\word{u} = \langle \word{u}^R \rangle_i$.
Therefore $\word{u}_1 = \word{u}_{n - i}, \word{u}_2 = \word{u}_{n - i - 1} \hdots \word{u}_{n - i} = \word{u}_1$ and $\word{u}_{n - i + 1} = \word{u}_n \hdots \word{u}_{1} = \word{u}_{n - i + 1}$.
This splits $\word{u}$ into 2 subwords, $\word{s}$ and $\word{t}$, where $\word{s} = \word{u}_{[1,n - i]}$ and $\word{t} = \word{u}_{[n - i + 1,n]}$ where $\word{s} = \word{s}^R$ and $\word{t} = \word{t}^R$.
Note that $\word{s}_1 \word{t} \word{s}_{n - i} = \word{s}_{n - i} \word{t}^R \word{s}_1$ and $\word{t}_1 \word{s} \word{t}_{n - i} = \word{t}_{n - i} \word{s}^R \word{t}_1$.

To show the structural claim, there are two cases to consider depending on the value of $i$ and $\frac{n}{2}$.
If $i$ is odd then lengths of $\word{s}$ and $\word{t}$ are even.
Two new words $\word{s}'$ and $\word{t}'$ are defined where $\word{s}' = \word{s}_{[(i/2) + 1,i]} \word{t}_{[1,(n-i)/2]}$ and $\word{t}' = \word{t}_{[(n - i)/2 + 1, n - i]} \word{s}_{[1,i/2]}$.
By the definition of $\word{s}$ and $\word{t}$, $\word{s}'^R = \word{t}_{[1,(n - i)/2]}^R \word{s}_{[i,i/2 + 1]}^R = \word{t}_{[(n - i)/2 + 1, n - i]} \word{s}_{[1,i/2 + 1]} = \word{t}'$.
Therefore this word can be rotated to a word of the form $\word{\phi} \word{\phi}^R$.

If $i$ is even then the lengths of $\word{s}$ and $\word{t}$ are odd.
As before 2 words $\word{s}'$ and $\word{t}'$ are constructed of length $\frac{n}{2} - 1$ where $\word{s}' = \word{s}_{[i/2 + 1, i]} \word{t}_{[1, (n - i)/2 - 1]}$ and $\word{t}' = \word{t}_{\frac{n - i}{2} + 1} \hdots \word{t}_{n - i} \word{s}_1 \hdots \word{s}_{\frac{i}{2} + 1}$.
As before, $\word{s}'^R = \word{t}_{[1,(n - i]/2 - 1]}^R \word{s}^R_{[i/2 + 1, i]}  = \word{t}_{[(n -i )/ 2 + 1, n - i]} \word{s}_{[1,i/2 - 1]} = v'$.
Letting $x = \word{t}_{\frac{n - i}{2}}$ and $y = \word{s}_{\frac{i}{2}}$, then there is some rotation of $\word{u}$ of the form $x \word{\phi} y \word{\phi}^R$.
 \end{proof}

\begin{proposition}
\label{prop:shared_even_form}
The word $\word{u} \in \Sigma^*$ equals both $x \word{\phi} y \word{\phi}^R = \word{\psi} \word{\psi}^R$ if and only if $\word{u} = x^n$.
\end{proposition}

\begin{proof}
Starting with $x \word{\phi} y \word{\phi}^R = \word{\psi} \word{\psi}^R$ as $x \word{\phi} = \word{\psi}$, $x \word{\phi} y \word{\phi}^R = x \word{\phi} \word{\phi}^R x$.
This implies $\word{\phi}_1 = x$ allowing this to be rewritten as $x x \word{\phi}' \word{\phi}'^R x x = x \word{\phi} y \word{\phi}^R$, implying that $\word{\phi}'_1 = x$.
Repeating this gives $x \word{\phi} y \word{\phi}^R = x x x \hdots x$.
 \end{proof}

\begin{proposition}
\label{prop:even_one_phrom}
For an even length palindromic necklace $\necklace{a}$ there are at most two words $\word{w},\word{u} \in \necklace{a}$ where either (1) $\word{w}$ and $\word{u}$ are of the form $x \word{\phi} y \word{\phi}^R$ where $x,y \in \Sigma$ and $\word{\phi} \in \Sigma^{(n/2) - 1}$ or (2) $\word{w}$ and $\word{u}$ are of the form $\word{\phi} \word{\phi}^R$ where $\word{\phi} \in \Sigma^{n/2}$.
\end{proposition}

\begin{proof}
From Proposition \ref{prop:even_phorm} there must be at least 1 word of either form.
Proposition \ref{prop:shared_even_form} shows that a word may only be of the form $x \word{\phi} y \word{\phi}^R$ and $\word{\psi} \word{\psi}^R$ if and only if $\word{w} = x^n$.
Let $\word{w}$ and $\word{v}$ be two words such that $\word{w},\word{v} \in \necklace{a}$ and $\word{w} \neq \word{v}$ where $\necklace{a}$ is a necklace of even length.
There are two cases based on the form of $\word{w}$ and $\word{v}$.

\textbf{Case 1:}
$\word{w} = x \word{\phi} y \word{\phi}^R$, $\word{v} = a \word{\psi} b \word{\psi}^R$, $\word{v} = \langle \word{w} \rangle_r$.
Let $r$ be the smallest rotation where $\langle \word{w} \rangle_r \neq \word{w}$ and $\langle \word{w} \rangle_r = a \word{\psi} b \word{\psi}^R$.
Therefore $\word{v}_i = \word{v}_{n - i + 1} = \word{w}_{n - i + r} = \word{w}_{n - n + i - r} = \word{w}_{i - r} = \word{v}_{n + i - 2r} = \word{v}_{i - 2r}$.
Therefore, $\word{v}_{i} = \word{v}_{i + 2r} = \word{v}_{i + 4r} = \hdots = \word{v}_{i}$.
Therefore $\word{w}$ has a period of no more than $p = GCD(2r,n)$.
If $GCD(2r,n) \leq 2r$, then the period must be no more than $r$.
If the period is $r$ then $\word{w} = \word{v}$, contradicting the assumption that they are not equal.
Otherwise, $\langle\word{w}\rangle_{r} = \langle\word{w}\rangle_{r - p}$, contradicting the assumption that $r$ is the smallest rotation for which the rotation of $\word{w}$ equals $x \word{\psi} y \word{\psi}^R$, for some arguments of $x,y \in \Sigma$ and $\word{\psi} \in \Sigma^*$.
Therefore the period must be $2r$.
Hence let $r > s$ be some rotation such that $\word{w} \neq \langle\word{w}\rangle_{s} \neq \langle\word{w}\rangle_{r}$.
As $\langle\word{w}\rangle_r = \word{w}^R, \langle\word{w}\rangle_{s + r} = (\langle\word{w}\rangle_{s})^R$.
As the period is $2r$, if $s + r > 2r$ then the rotation $s - r$ is equivalent to the rotation by $s$ contradicting the assumption that $r$ is the smallest rotation for which $\langle\word{w}\rangle_r = x \word{\psi} y \word{\psi}^R$, for some arguments of $x,y \in \Sigma$ and $\word{\psi} \in \Sigma^*$.
Therefore the only word satisfying $\word{v}_i = \word{v}_{i -2r}$ is when $r = \frac{n}{2}$, making $\word{v} = y \word{\phi}^R x \word{\phi}$.

\textbf{Case 2:} $\word{w} = \word{\phi} \word{\phi}^R$, $\word{v} = \word{\psi} \word{\psi}^R$, $\word{v} = \langle \word{w} \rangle_r $.
For the sake of contradiction, let $r$ be the smallest rotation such that $\word{w} \neq \langle \word{w} \rangle_r$ and $\langle \word{w} \rangle_r = \word{\psi} \word{\psi}^R$.
Therefore $\word{v}_i = \word{w}_{i + r \bmod n}$, further $\word{w}_i = \word{v}_{n + i - r \bmod n}$, $\word{w}_i = \word{w}_{n - i + 1}$ and $\word{v}_i = \word{v}_{n - i + 1}$.
These equations can be rearranged to give $\word{v}_i = \word{v}_{n - i + 1} = \word{w}_{r + n - i - 1} = \word{w}_{n - r - n + i + 1 - 1} = \word{w}_{i - r} = \word{v}_{i - 2r} = \word{v}_i$.
Repeated application of $\word{v}_i = \word{v}_{i - 2r} = \word{v}_{i - 4r} = \hdots = \word{v}_{i - s\cdot r}$ shows that $\word{w}$ must have a period of no more than $p = GCD(2r,n)$.
Therefore $\word{w}$ can be rewritten as $\word{u}^{n/p} = \word{\phi} \word{\phi}^R$.
If $\frac{n}{p}$ is even then $\word{u} = \word{u}^{R}$.
Assume for the sake of contradiction that there is some rotation $t$ such that $r < t < 2r, \word{w} \neq \langle \word{w} \rangle_t \neq \word{v}$ and $\langle \word{w} \rangle_t$ is of the form $\word{\phi} \word{\phi}^{R}$.
If $\word{u} = \word{u}^R$, then $\langle \word{u} \rangle_t = (\langle \word{u} \rangle_t)^R$.
Hence the rotation by $2r - t$ is equivalent to the rotation by $t$, contradicting the assumption that $r$ is the smallest rotation.
If the period of $\word{w}$ is smaller than $2r$ it must be a factor of $r$, hence $\langle \word{w} \rangle r = \word{w}$ contradicting the assumption that $\word{w} \neq \word{v}$.
Therefore the period must be $2r$, implying that if $\word{w} = \word{\phi} \word{\phi}^R$ then $\word{v} = \word{\phi}^R \word{\phi}$.
If $\frac{n}{p}$ is odd then as $\word{u}^{n/p} = \word{\phi} \word{\phi}^R, \word{u}_{[1,r]} = \word{u}_{r + 1,2r}$.
Therefore the period is at most $r$, contradicting the assumption that $p = GCD(2r, n)$.
In this case the arguments from the even case apply again.
\end{proof}

\noindent
Propositions \ref{prop:even_phorm}, \ref{prop:shared_even_form} and \ref{prop:even_one_phrom} show that every palindromic necklace of even length has 1 or 2 words of either the form $x \word{\phi} y \word{\phi}^R$ or $\word{\phi} \word{\phi}^{R}$.
To count the number of words of each form, the problem is split into two sub problems, counting words of the form $x \word{\phi} y \word{\phi}^R$ and counting the number of words of the form $\word{\phi} \word{\phi}^R$.
This is done using the same basic ideas as in Section \ref{subsec:odd_length_palindromic}.
Two new sets $\mathcal{PE}(\word{v})$ and $\mathcal{PS}(\word{v})$ are introduced, serving the same function as $\mathcal{PO}(\word{v})$ for words of the from $x \word{\phi} y \word{\phi}^R$ and $\word{\phi} \word{\phi}^R$ respectively.
\begin{align*}
&\mathcal{PE}(\word{v}) := \left\{\word{w} \in \Sigma^n : \word{w} = x \word{\phi} y \word{\phi}^R, \text{ where }\langle\word{w}\rangle > \word{v}, \word{\phi} \in \Sigma^{(n/2) - 1}, x,y, \in \Sigma\right\}\\
&\mathcal{PS}(\word{v}) := \left\{\word{w} \in \Sigma^n : \word{w} = \word{\phi} \word{\phi}^R, \text{ where }\langle\word{w}\rangle > \word{v}, \word{\phi} \in \Sigma^{(n/2) - 1}\right\}
\end{align*}

\noindent
Unlike the set $\mathcal{PO}(\word{v})$ in Section \ref{subsec:odd_length_palindromic} the sets $\mathcal{PE}(\word{v})$ and $\mathcal{PS}(\word{v})$ do not correspond directly to bracelets greater than $\word{v}$.
For notation let $\mathcal{GE}(\word{v})$ and $\mathcal{GS}(\word{v})$ denote the number of bracelets greater than $\word{v}$ of the form $x \word{\phi} y \word{\phi}^R$ and $\word{\phi} \word{\phi}^R$ respectively.
The number of even length necklaces greater than $\word{v}$ equals $\mathcal{GE}(\word{v}) + \mathcal{GS}(\word{v}) - (k - \word{v}_1)$, where $k - \word{v}_1$ denotes the number of symbols in $\Sigma$ greater than $\word{v}_1$.
Before showing how to compute the size of these sets, it is useful to first understand how they are used to compute the rank amongst even length palindromic necklaces.
Lemmas \ref{lem:PS_conversion} and \ref{lem:PE_conversion} shows how to covert the cardinalities of these sets into the number of even length palindromic necklaces smaller than $\word{v}$.
The main idea is to use the observations given by Propositions \ref{prop:even_phorm} and \ref{prop:even_one_phrom} to determine how many even length palindromic necklaces have either one or two words of the form $x \word{\phi} y \word{\phi}^R$ or $\word{\phi} \word{\phi}^R$.

\begin{proposition}
Let $l = \frac{n + 2}{4}$ if $\frac{n}{2}$ is odd or $l = \frac{n}{4}$ if $\frac{n}{2}$ is even.
The number of even length palindromic necklaces is given by $\frac{1}{2}\left(k^{n/2}(k + 2) + k^{l}\right) - k$.
\label{prop:even_no_necklaces}
\end{proposition}

\begin{proof}
First consider the number of words of the form $x \word{\phi} y \word{\phi}^R$.
Let $\word{w},\word{u} \in \necklace{w}$ be a pair of words of the form $x \word{\phi} y \word{\phi}^R$ such that $\word{w} \neq \word{u}$ and $\langle \word{w} \rangle_r = \word{u}$.
Following Proposition \ref{prop:even_one_phrom}, if $2r < n$ and $\frac{n}{2r}$ is odd, then $\word{w} = x \word{\phi} y \word{\phi}^R = \word{\psi}^t$ for some word $\word{\psi}$ of even length and $t = \frac{n}{2r}$.
Therefore $\word{\psi} =\word{\psi}^R$ and further $\word{\psi} = \langle\word{\psi}\rangle_{r}$, therefore there will only be a single word of the form $x \word{\phi} y \word{\phi}$.
On the other hand if $2r < n$ and $\frac{n}{2r}$ is even then $\word{w} = x \word{\phi} y \word{\phi}^R = \word{\psi}^t$ for $t = \frac{n}{2r}$ and some word $\word{\psi}$ of length $2r$.
In this case, as $t$ must be at least 2, $x \word{\phi} y \word{\phi}^R = x \word{\phi} x \word{\phi}$, therefore $y = x$ and $\word{\phi} = \word{\phi}^R$.
Further as $\word{u} = (\word{\psi}_r)^t$, $\word{\psi} = \word{\psi}^R$, therefore $\word{u} = \word{w}$, hence there is only a single word of the form $x \word{\phi} y \word{\phi}^R$.
Therefore the period of $\word{w}$ must be $n$ and hence there are only two words of the form $x \word{\phi} y \word{\phi}^R$ if and only if $x \word{\phi} \neq y \word{\phi}^R$.

Using this basis, the number of even length palindromic necklaces with one words of the form $x \word{\phi} y \word{\phi}^R$ equals the number of words of the form $x \word{\phi} x \word{\phi}$.
This is equal to $k^{n/2}$.
As the number of words with 2 representations of the form $x \word{\phi} y \word{\phi}^R$ is $k^{(k/2) + 1}$, the number of necklaces with any word of the form $x \word{\phi} y \word{\phi}^R$ is $\frac{1}{2}\left(k^{n/2 + 1} + k^{n/2}\right)$.

Proposition \ref{prop:even_one_phrom} shows that, given $\word{w},\word{u} \in \necklace{w}$ of the form $\word{\phi} \word{\phi}^R$, $\word{w} \neq \word{u}$ if and only if $\word{u} = \langle\word{w}\rangle_{n/2}$ and $\word{\phi} \neq \word{\phi}^R$.
Therefore the number of necklaces with 1 word of the form $\word{\phi} \word{\phi}^R$ is equal to the number of values of $\word{\phi}$ for which $\word{\phi} = \word{\phi}^R$.
If $|\word{\phi}|$ is odd, this is equal to $k^{(n + 2)/4}$ and $k^{n/4}$ if $|\word{\phi}|$ is even.
Hence the number of necklaces with two representations of the form $\word{\phi}\word{\phi}^R$ is $\frac{1}{2}(k^{n/2} - k^l)$, where $l = \frac{n + 2}{4}$ if $\frac{n}{2}$ is odd or $l = \frac{n}{4}$ if $\frac{n}{2}$ is even.
Therefore the total number of necklaces with any word of the form $\word{\phi} \word{\phi}^R$ is $\frac{1}{2}\left(k^{n/2} + k^{l}\right)$.
Recalling from Proposition \ref{prop:shared_even_form} that a word is of both forms if and only if it is of the form $x^n$, there are $k$ necklaces that would be counted by both equations.
Therefore the total number of even length necklaces are $\frac{1}{2}\left(k^{n/2 + 1} + k^{n/2} + k^{n/2} + k^{l}\right) - k$.
\end{proof}

\begin{lemma}
The number of necklaces greater than $\word{v}$ containing at least one word of the form $x \word{\phi} y \word{\phi}^R$ is given by $GE(\word{v}) = \frac{1}{2}\left(|\mathcal{PE}(\word{v})| + \begin{cases}
|\mathcal{PO}(\word{v}_{[1,n/2]})| & \frac{n}{2} \text{ is odd.}\\
GE(\word{v}_{[1,n/2]}) & \frac{n}{2} \text{ is even.}
\end{cases}\right)$.
\label{lem:PE_conversion}
\end{lemma}

\begin{proof}
It follows that the number of necklaces of the form $x \word{\phi} y \word{\phi}^R$ that are greater than $\word{v}$ equals to the number of necklaces with one word of the form $x \word{\phi} y \word{\phi}^R$, plus the number of necklaces with two words of the form $x \word{\phi} y \word{\phi}^R$.
The number of words of the form $x \word{\phi} y \word{\phi}^R$ greater than $\word{v}$ equals the size of $\mathcal{PE}(\word{v})$.
As a necklace has only one word of the form $x \word{\phi} y \word{\phi}^R$ if and only if $x \word{\phi} = y \word{\phi}^R$.
This leaves the problem of counting the number of words of the form $x \word{\phi} x \word{\phi}$ in necklaces greater than $\word{v}$.
If $\frac{n}{2}$ is odd, then $\word{\phi}$ can be rewritten as $\word{\psi} \word{\psi}^R$.
In this case, the goal becomes to fine the number of words of the form $x \word{\psi} \word{\psi}^R x \word{\psi} \word{\psi}^R$ in bracelets greater than $\word{v}$, which equals $|\mathcal{PO}(\word{v}_{[1,n/2]})|$.
On the other hand, if $\frac{n - 2}{2}$ is odd then $\word{\phi}$ can be rewritten as $\word{\psi} y \word{\psi}^R$.
In this case, the goal becomes to fine the number of words of the form $x \word{\psi} y \word{\psi}^R x \word{\psi} y \word{\psi}^R$ in bracelets greater than $\word{v}$, which equals the number of words of the form $x \word{\psi} y \word{\psi}^R$ that are bracelets greater than $\word{v}$.
This is given by $GE(\word{v}_{[1,n/2]})$.
Therefore the total number of necklaces of the form $x \word{\phi} y \word{\phi}^R$ greater than $\word{v}$ is given by:

\[
GE(\word{v}) = \frac{1}{2}\left(|\mathcal{PE}(\word{v})| + \begin{cases}
|\mathcal{PO}(\word{v}_{[1,n/2]})| & \frac{n}{2} \text{ is odd.}\\
GE(\word{v}_{[1,n/2]}) & \frac{n}{2} \text{ is even.}
\end{cases}\right)
\]
\end{proof}

\begin{lemma}
The number of necklaces greater than $\word{v}$ containing at least one word of the form $\word{\phi} \word{\phi}^R$ is given by $GS(\word{v}) = \frac{1}{2}\left(|\mathcal{PE}(\word{v})| + \begin{cases} |\mathcal{PO}(\word{v})| & \frac{n}{2} \text{ is odd.}\\
GS(\word{v}_{[1,n/2]}) & \frac{n}{2} \text{ is even.} \end{cases}\right)$.
\label{lem:PS_conversion}
\end{lemma}

\begin{proof}
Similar to Lemma \ref{lem:PE_conversion}, this Lemma is proven in a combinatorial manner by looking at the two cases where there is only a single word of the form $\word{\phi} \word{\phi}^R$.
Recall that there is a single word of this form if and only if $\word{\phi} = \word{\phi}^R$.
Therefore, the number of necklaces with a single word of the form $\word{\phi} \word{\phi}^R$ equals the number of palindromic words of length $\frac{n}{2}$.
Hence if $\frac{n}{2}$ is even, the number of such words is $GS(\word{v}_{[1,n/2]})$.
On the other hand, if $\frac{n}{2}$ is odd, the number of such words is $|\mathcal{PO}(\word{v})|$.
Using the same arguments as in Proposition \ref{prop:even_no_necklaces}:
\[
GE(\word{v}) = \frac{1}{2}\left(|\mathcal{PE}(\word{v})| + \begin{cases} |\mathcal{PO}(\word{v})| & \frac{n}{2} \text{ is odd.}\\
GS(\word{v}_{[1,n/2]}) & \frac{n}{2} \text{ is even.} \end{cases}\right)
\]
\end{proof}

\noindent
\textbf{High Level Idea for the Even Case:}
Lemmas \ref{lem:PE_conversion} and \ref{lem:PS_conversion} show how to use the sets $\mathcal{PS}(\word{v})$ and $\mathcal{PE}(\word{v})$ to get the number of necklaces of the form $x \word{\phi} y \word{\phi}^R$ and $\word{\phi} \word{\phi}^R$ respectively.
This leaves the problem of computing the size of both sets.
This is achieved in a manner similar to the one outlined in Section \ref{subsec:odd_length_palindromic}.
At a high level the idea is to use two trees analogous to $\mathcal{TO}(\word{v})$ as defined in Section \ref{subsec:odd_length_palindromic}.
The tree $\mathcal{TE}(\word{v})$ is introduced to compute the cardinality of $\mathcal{PE}(\word{v})$ and the  tree $\mathcal{TS}(\word{v})$ is introduced to compute the cardinality of $\mathcal{PS}(\word{v})$.
As in Section \ref{subsec:odd_length_palindromic}, the trees $\mathcal{TE}(\word{v})$ and  $\mathcal{TS}(\word{v})$ contain every prefix of a word in $\mathcal{PS}(\word{v})$ or $\mathcal{PE}(\word{v})$ respectively.
The leaf vertices of these trees correspond to the words in these sets.

To compute the size of $\mathcal{PE}(\word{v})$ using $\mathcal{TE}(\word{v})$, the same approach as in Section \ref{subsec:odd_length_palindromic} is used.
A word $\word{u}$ of length less than $\frac{n}{2}$ is a prefix of some word in $\mathcal{PE}(\word{v})$ if and only if no subword of $(\word{u}_{[1,|\word{u}| - 1]})^R \word{u}$ is less than the prefix of $\word{v}$ of the same length.
This is slightly different from the odd case, where $\word{u}\in\mathcal{PE}(\word{v})$ if and only if there is no subword of $\word{u}^R \word{u}$ smaller than the prefix of $\word{v}$ of the same length.
To account for this difference the sets $\mathbf{PE}(\word{v},i,j,\word{s})$ are introduced as analogies to the sets $\mathbf{PO}(\word{v},i,j,\word{s})$.

\begin{definition}
Let $i \in [\frac{n + 1}{2}], j \in [2i]$ and $\word{s} \sqsubseteq_{2i} \word{v}$.
The set $\mathbf{PE}(\word{v},i,j,\word{s})$ contains every word $\word{u} \in \mathcal{TE}(\word{v})$ of length $i$ where (1) the longest suffix of $(\word{u}_{[1,i -1]})^R \word{u}_{[1,i]}$  which is a prefix of $\word{v}$ has a length of $j$ and (2) the word $(\word{u}_{[1,i - 1]})^R \word{u}_{[1,i]}$ is strictly bounded by $\word{s} \sqsubseteq_{2i - 1} \word{v}$.
\label{def:PE_4_params}
\end{definition}

\noindent
As in Section \ref{subsec:odd_length_palindromic}, the size of $\mathbf{PE}(\word{v},i,j,\word{s})$ is computed via dynamic programming.
The array $SizePE$ is introduced, storing the size of $\mathbf{PE}(\word{v},i,j,\word{s})$ for every value of $i \in \left[\frac{n}{2}\right], j \in [2i - 1]$ and $\word{s} \sqsubseteq_{2i - 1} \word{v}$.
Let $SizePE$ be and $n\times n \times n$ array such that $SizePE[i,j,\word{s}] = |\mathbf{PE}(\word{v},i,j,\word{s})|$.
Lemma \ref{lem:PE_Properties} shows that the techniques used in Lemma \ref{lem:SizeOfPO(i,j,s)} can be used to compute $SizePE$ in $O(k \cdot n^3 \log(n))$ time.
This is done by proving that the properties established by Lemma \ref{lem:PO_cartesian} regarding the relationship between the sets $\mathbf{PO}(\word{v},i,j,\word{s})$ also hold for the sets $\mathbf{PE}(\word{v},i,j,\word{s})$.
As words in $\mathcal{PS}(\word{v})$ are of the form $\word{\phi} \word{\phi}^R$, a word $\word{u}$ is in $\mathcal{TS}(\word{v})$ if and only if no subword of $\word{u}^R \word{u}$ is less than the prefix of $\word{v}$ of the same length.
Note that this corresponds to the same requirement as the odd case.
As such the internal vertices in the tree $\mathcal{TS}(\word{v})$ may be partitioned in the same way as those of $\mathcal{TO}(\word{v})$.
Lemma \ref{lem:PS} shows how to convert the array $SizePO$ as defined is Section \ref{subsec:odd_length_palindromic} to the size of $\mathcal{PS}(\word{v})$.

\begin{lemma}
Given $\word{u},\word{w} \in \mathbf{PE}(\word{v},i,j,\word{s})$ and $x \in \Sigma$.
If $\word{u} x \in \mathbf{PE}(\word{v},i+1,j',\word{s}')$ then $\word{v} x \in \mathbf{PE}(\word{v},i+1,j',\word{s}')$.
Further the values of $j'$ and $\word{s}'$ can be computed in constant time from the values of $j,\word{s}$ and $x$.
Therefore the array $SizePE[i,j,\word{s}]$ can be computed for every value $i \in \left[\frac{n}{2}\right], j \in [2i - 1]$ and $\word{s} \sqsubseteq_{2i - 1} \word{v}$ in $O(k \cdot n^3 \cdot \log(n))$ time.
\label{lem:PE_Properties}
\end{lemma}

\begin{proof}
Note that these are the same properties as proven in Lemma \ref{lem:PO_cartesian}.
As the arguments $j$ and $\word{s}$ serve the same function for both $\mathbf{PE}(\word{v},i,j,\word{s})$ and $\mathbf{PO}(\word{v},i,j,\word{s})$, the arguments from Lemma \ref{lem:PO_cartesian} can be applied directly to this setting.

Following the above arguments, the techniques employed in Lemma \ref{lem:SizeOfPO(i,j,s)} can be applied to computing the value of $PE[i,j,\word{s}]$ for every argument $i \in \left[\frac{n}{2}\right], j \in [2i - 1]$ and $\word{s} \sqsubseteq_{2i - 1} \word{v}$.
The only modification needed is to account for the change the form of the words in $\mathbf{PE}(\word{v},i,j,\word{s})$ versus those in $\mathbf{PO}(\word{v},i,j,\word{s})$.
As the words in $\mathbf{PE}(\word{v},i,j,\word{s})$ have the form $\word{\phi}^R x \word{\phi}$, rather than $\word{\phi}^R \word{\phi}$, the set $\mathbf{PE}(\word{v},i,j,\word{s})$ represents words of length $2i - 1$.
\end{proof}

\begin{lemma}
Let $\word{v} \in \Sigma^{n}$.
The size of $\mathcal{PE}(\word{v})$ can be computed in $O(k \cdot n^3 \cdot \log(n))$ time.
\label{lem:PE}
\end{lemma}

\begin{proof}
Note that for every word $\word{w} \in \mathcal{PE}(\word{v})$, either $\word{w}_{[1,(n/2)]} \in \mathbf{PE}(\word{v},\frac{n}{2},j,\word{s})$ or $(\word{w}_{[2,(n/2) - 1]})^R \word{w}_{[1,(n/2) - 1]} \sqsubseteq \word{v}$.
Following the arguments in Lemmas \ref{lem:Z(v,i,j)Computaiton} and \ref{lem:SizeOfPO}, the number of words $\word{w} \in \mathcal{PE}(\word{v})$ where $\word{w}_{[1,(n/2)]} \in \mathbf{PE}(\word{v},\frac{n}{2},j,\word{s})$ for some given values of $j \in \left[n - 1\right]$ and $\word{s} \sqsubseteq_{n - 1} \word{v}$ is equal to the number of symbols $x \in \Sigma$ where $\langle  \word{w}_{[1,(n/2)]} x \word{w}_{[1,(n/2) - 1]}^R\rangle > \word{v}$.
Using the same techniques laid out in Lemma \ref{lem:Z(v,i,j)Computaiton}, the set of such symbols can be computed in $O(k \cdot n)$ time.
It follows that given the array $PE$, the number of words $\word{w} \in \mathcal{PO}(\word{v})$ where $(\word{w}_1,\word{w}_2, \hdots. \word{w}_{n/2}) \in \mathbf{PE}(\word{v},\frac{n}{2},j,\word{s})$ can be computed in $O(n^2 \cdot k)$ operations by checking every combination of $j \in \left[\frac{n}{2}\right], \word{s} \in \mathbf{S}(\word{v},n - 1)$ and $x \in \Sigma$.

Similarly if $\word{w} \in \mathbf{S}(\word{v},n - 1)$, then $\word{w} x \in \mathcal{PE}(\word{v})$ if and only if $\word{w} = \word{\phi}^R x \word{\phi}$ and $\word{w} \langle x \rangle > \word{v}$.
Each subword $\word{s} \in \mathbf{S}(\word{v}, n - 1)$ may be checked in $O(n^2)$ operations by first checking that $\word{s} = \word{s}^R$, then finding the smallest rotation of $\word{s} x$ and comparing it to $\word{v}$.
As there are $n$ words in $\mathbf{S}(\word{v},n - 1)$ and $k$ symbols in $\Sigma$, this it will take $O(n^3 \cdot k)$ operations.
Computing the arrays $PE,WX$ and $XW$ will take $O(n^3 \cdot k \cdot \log(n))$ time, hence the total complexity is $O(n^3 \cdot k \cdot \log(n))$.
\end{proof}

The size of $\mathcal{PS}(\word{v})$ is calculated in a similar manner.
As the words in $\mathcal{PS}(\word{v})$ are of the form $\word{\phi} \word{\phi}^R$, the prefixes of length $i$ correspond to subwords of length $2i$ with the form $\word{u}^R \word{u}$.
Note that these are the same as the prefixes used in Section \ref{subsec:odd_length_palindromic} for odd length palindromic necklaces.
As such, the sets $\mathbf{PO}(\word{v},i,j,\word{s})$ are used to partition internal vertices of the tree $\mathcal{TS}(\word{v})$.
Lemma \ref{lem:PS} shows how to use these sets to compute the size of $\mathcal{PS}(\word{v})$.

\begin{lemma}
Let $\word{v} \in \Sigma^n$.
The size of $\mathcal{PS}(\word{v})$ can be computed in $O(k \cdot n^3 \cdot \log(n))$ time.
\label{lem:PS}
\end{lemma}

\begin{proof}
For every word $\word{w} \in \mathcal{PS}(\word{v})$ there are two cases to consider:
\begin{itemize}
    \item \textbf{Case 1:} $(\word{w}_{[1,(n/2) - 1]})^R\word{w}_{[1,(n/2) - 1]} \sqsubseteq \word{v}$.
    \item \textbf{Case 2:} There exists some set $\mathbf{PO}(\word{v}, \frac{n}{2} - 1, j, \word{s})$ such that $\word{w}_{[1,(n/2) - 1]} \in \mathbf{PS}(\word{v}, \frac{n}{2} - 1, j, \word{s})$.
\end{itemize}

The number of words in the first case can be computed by considering every subword $\word{s} \in \mathbf{S}(\word{v},n - 2)$ and $z \in \Sigma$ where $\word{s} = \word{s}^R$ and $\langle z \word{s} z \rangle > \word{v}$.
Note both of the above conditions can be checked in at most $O(n)$ operations.
If both conditions hold, then $\word{s}$ and $z$ correspond to exactly one word in $\mathcal{PS}(\word{v})$.
As there are $n$ possible values of $\word{s}$ and $k$ values of $z$ therefore the number of words in this case can be computed in $O(n^2 \cdot k)$ operations.

The number of words in the second case can be computed by considering every vale of $j \in [n - 2], \word{s} \in \mathbf{S}(\word{v},n - 2)$ and $z \in \Sigma$.
Let $\word{w}_{[1,(n/2)-1]} \in \mathbf{PS}(\word{v}, \frac{n}{2} - 1, j, \word{s})$.
The word $z (\word{w}_{[1,(n/2)-1]})^R \word{w}_{[1,(n/2)-1]} z \in \mathcal{PS}(\word{v})$ if and only if $\langle z (\word{w}_{[1,(n/2)-1]})^R \word{w}_{[1,(n/2)-1]} z\rangle > \word{v}$.
This is the case if and only if $\word{v}_{[1,j]} z z \word{s} > \word{v}$ which can be checked in $O(n)$ time.
If $\word{v}_{[1,j]} z z \word{s} > \word{v}$, then there are $PS[\frac{n}{2} - 1,j,\word{s}]$ prefixes in $\mathbf{PS}(\word{v},i,j,\word{s})$ such that $\langle z (\word{w}_{[1,(n/2)-1]})^R \word{w}_{[1,(n/2)-1]} z\rangle > \word{v}$.
As there are $n$ values of $j$ and $\word{s}$ and $k$ values of $z$ the number of words in this case can be computed in $O(n^3 \cdot k)$ operations.
Finally, in order to compute this case in $O(n^3 \cdot k)$ steps, the array $PS$ must be precomputed, requiring $O(k \cdot n^3 \cdot \log(n))$ operations.
Therefore the total complexity is $O(k \cdot n^3 \cdot \log(n))$.
\end{proof}

Combining Lemmas \ref{lem:PE} and \ref{lem:PS} with Lemmas \ref{lem:PE_conversion} and \ref{lem:PS_conversion} provides the tools to compute the rank of \word{v} among even length palindromic necklaces.
Lemma \ref{lem:complexity_even_overall} shows how to combine these values to get the rank of $\word{v}$ among even length palindromic necklaces.

\begin{lemma}
\label{lem:complexity_even_overall}
The rank of $\word{v} \in \Sigma^{n}$ among even length palindromic necklaces can be computed in $O(k \cdot n^3 \cdot \log(n)^2)$ time.
\end{lemma}

\begin{proof}
From Proposition \ref{prop:even_no_necklaces}, the number of even length palindromic necklaces is equal to $\frac{1}{2}\left(k^{n/2 + 1} + 2k^{n/2} + k^l\right) - k$, where $l = \frac{n + 2}{4}$ if $\frac{n}{2}$ is odd, or $l = \frac{n}{4}$ if $\frac{n}{2}$ is even.
Lemma \ref{lem:PE_conversion} provides an equation to count the number of necklaces greater than $\word{v}$ containing at least one word of the form $x \word{\phi} y \word{\phi}^R$.
The equation given by Lemma \ref{lem:PE_conversion} requires the size of $\mathcal{PE}(\word{v})$ to be computed, needing at most $O(k \cdot n^3 \cdot \log(n))$ operations, and either $|\mathcal{PE}(\word{v}_{[1,n/2]})|$ or $GE(\word{v}_{[1,n/2]})$.
As both $|\mathcal{PE}(\word{v})|$ and $|\mathcal{PO}(\word{v})|$ require $O(k \cdot n^3 \cdot \log(n))$ operations, the total complexity comes from the number of such sets that must be considered.
As the prefixes of $\word{v}$ that need to be computed is no more than $\log_2(n)$, the total complexity of computing $GE(\word{v})$ is $O(k \cdot n^3 \cdot \log^2(n))$.
Similarly as the complexity of computing $\mathcal{PS}(\word{v})$ is $O(k \cdot n^3 \cdot \log(n))$, the complexity of computing $GS(\word{v})$ is $O(k \cdot n^3 \cdot \log^2(n))$.
\end{proof}


\begin{theorem}
\label{thm:pallindromic_comlexity}
Give a word $\word{v} \in \Sigma^n$, the rank of $\word{v}$ with respect to the set of palindromic necklaces, $RP(\word{v})$, can be computed in $O(k \cdot n^3 \cdot \log^2(n))$ time.
\end{theorem}

\begin{proof}
The number of odd length palindromic necklaces is given by Proposition \ref{prop:num_odd_necklaces} as $k^{(n - 1)/2}$.
Lemma \ref{lem:SizeOfPO} shows that the size of set $\mathcal{PO}(\word{v})$, corresponding to the number of odd length palindromic bracelets, can be computed in $O(k \cdot n^3 \cdot \log(n))$ time.
By subtracting the size of $\mathcal{PO}(\word{v})$ from $k^{(n - 1)/2}$, the rank of $\word{v}$ can be computed in $O(k \cdot n^3 \cdot \log(n))$ time.
Lemma \ref{lem:complexity_even_overall} shows that of $RP(\word{v})$ can be computed in $O(k \cdot n^3 \cdot \log^2(n))$ time if the length of $\word{v}$ is even.
Hence the total complexity is $O(k \cdot n^3 \cdot \log^2(n))$.
\end{proof}

\section{Enclosing Bracelets}
\label{sec:Stv}

Following Lemma \ref{lem:ranking_bracelets} and Theorem \ref{thm:pallindromic_comlexity}, the remaining problem is counting the number of enclosing words.
This section will provide a technique to count the number of necklaces enclosing some word $\word{v}$.
As in the palindromic case, the structure of these words will first be analysed so that a more efficient algorithm can be derived.

\begin{proposition}
The bracelet representation of every bracelet $\bracelet{w}$ enclosing the word $\word{v} \in \Sigma^n$ can be written as  $\word{v}_{[1,i]} x \word{\phi}$ where; $x \in \Sigma$ is a symbol that is strictly smaller than $\word{v}_{[i + 1]}$, and $\word{\phi} \in \Sigma^*$ is a word such that every rotation of $(\word{v}_{[1,i]} x \word{\phi})^R$ is greater than $\word{v}$.
\label{prop:enclose_form}
\end{proposition}

\begin{proof}
For the sake of contradiction let $\bracelet{w}$ be a bracelet enclosing $\word{v}$ such that the bracelet representation of $\bracelet{w}$, $\word{a}$ can not be written as $\word{v}_{[1,i]} x \word{\phi}$.
Let $\word{b} = \langle \word{a}^R \rangle$.
By the definition of an enclosing necklace, $\word{a} < \word{v} < \word{b}$.
If $\word{a}_1 < \word{v}_1$, then $\word{b}_1 < \word{v}_1$.
Similarly if $\word{b}_1 > \word{v}_1$ then $\word{a}_1 > \word{v}_1$.
Hence $\word{a}_1 = \word{v}_1 = \word{b}_1$.
Therefore there exists some non zero value of $i$ such that $\word{a}_{[1,i]} = \word{v}_{[1,i]}$.

Let $i$ be the length of the longest shared prefix of $\word{v}$ and $\word{a}$, i.e. the largest value such that $\word{v}_{[1,i]} = \word{a}_{[1,i]}$.
If the symbol $\word{a}_{i + 1} > \word{v}_{i + 1}$ $\word{a} > \word{v}$ contradicting the assumption that $\word{a} < \word{v}$.
Similarly if $\word{a}_{i + 1} = \word{v}_{i + 1}$, there is a longer shared prefix.
Therefore $\word{a}_{i + 1} < \word{v}_{i + 1}$.

As this word can be written as $\word{v}_{[1,i]} x \word{\phi}$, it must be assumed that some rotation of $(\word{v}_{[1,i]} x \word{\phi})^R$ is less than or equal to $\word{v}$.
If this is the case, $\bracelet{w}$ does not enclose $\word{v}$, as both necklace classes are smaller than or equal to $\word{v}$.
Therefore the bracelet representation of every bracelet $\bracelet{w}$ enclosing the word $\word{v} \in \Sigma^n$ can be written as stated.
\end{proof}

\begin{proposition}
Given a bracelet $\bracelet{w}$ enclosing the word $\word{v} \in \Sigma^n$ of the form $\word{v}_{[1,j]} x \word{\phi}$ as given in Proposition \ref{prop:enclose_form}.
The value of $x$ must be greater than or equal to $\word{v}_{[(j + 1) \bmod l]}$ where $l$ is the length of the longest Lyndon word that is a prefix of $\word{v}_{[1,j]}$.
\end{proposition}

\begin{proof}
For the sake of contradiction assume that $x <  \word{v}_{[(j + 1) \bmod l]}$.
Following Theorem 2.1 due to Cattell et. al. \cite{Cattell2000}, the subword $\word{v}_{[j - (j \bmod l), j]} = \word{v}_{[1,j \bmod l]}$.
Therefore if $x < \word{v}_{j + 1 \bmod l}$ then the subword $\word{v}_{[1,j \bmod l]} x < \word{v}_{[1,l]}$.
In this case, there is a smaller rotation of $\word{v}_{[1,j]} x \word{\phi}$, contradicting our assumption the $\word{v}_{[1,j]} x \word{\phi}$ is the smallest rotation.
Hence $x$ must be greater than or equal to $\word{v}_{j + 1 \bmod l}$.
\end{proof}

\noindent
\textbf{High Level Idea for the Enclosing Case:}
Similar to Sections \ref{subsec:odd_length_palindromic} and \ref{subsec:even_length_palindromic}, the main idea is to use the structure given in Proposition \ref{prop:enclose_form} as a basis for counting the number of enclosing bracelets.
For each value of $i$ and $x$, the number of possible values of $\word{\phi}$ are counted.
This is done in a recursive manner, working backwards from the last symbol.
For each combination of $i$ and $x$, the key properties to observe are that (1) every suffix of $\word{\phi}$ must be greater than or equal to $\word{v}_{[1,i]} x$ and (2) every rotation of $\word{\phi}^R x \word{v}_{[1,i]}^R$ is greater than $\word{v}$.

These observations are used to create a tree, $\mathcal{TEN}(\word{v},i,x)$, where each vertex represents a suffix of some possible value of $\word{\phi}$. 
Equivalently, the vertices of $\mathcal{TEN}(\word{v},i,x)$ can be thought of as representing the prefixes of $\word{\phi}^R$.
The leaf vertices of $\mathcal{TEN}(\word{v},i,x)$ represent the possible values of $\word{\phi}$.
As in Section \ref{sec:Ptv}, each layer of $\mathcal{TEN}(\word{v},i,x)$ is grouped into sets based on the lexicographical value of the reflection of the suffixes, and the prefixes of the suffixes.
Let $t \in [|\word{w}| - i]$, $j \in [t + i + 1]$ and $\word{s} \sqsubseteq_{t + i + 1} \word{v}$.
For the $t^{th}$ layer of $\mathcal{TEN}(\word{v},i,x)$, the set $\mathcal{E}(\word{v},i,x,j,\word{s})$ is introduced containing a subset of the vertices at layer $t$.
The idea is to use the values of $j$ and $\word{s}$ to divide the prefixes at layer $t$ by lexicographic value and suffix respectively.
Let $\word{u} \in \mathcal{E}(\word{v},i,x,j,\word{s})$ be a suffix of some word $\word{w}$ such that $\word{v}_{[1,i]} x \word{w}$ is a bracelet enclosing $\word{v}$.
To ensure that the necklace represented by the reflection is strictly greater than $\word{v}$, $j$ is used to track the longest prefix of $\word{u}^R$ that is a prefix of $\word{v}$.
To ensure that there is no rotation of $x \word{v}_{[1,i]}^R \word{w}^R$, the subword $\word{s} \sqsubseteq_t \word{v}$ is used to bound the value of $\word{u}^R$.
Formally, $\mathcal{E}(\word{v},i,x,j,\word{s})$ contains every suffix $\word{u} \in \mathcal{TEN}(\word{v},i,x)$ of length $i$ where (1) the longest prefix of $\word{u}^R$ that is also a prefix of $\word{v}$ and (2) the subword $\word{s} \sqsubseteq_t \word{v}$ bounds $\word{u}^R$.

As in Section \ref{sec:Ptv} the number of leaf vertices are calculated by determining the size of the sets $\mathcal{E}(\word{v},i,x,j,\word{s})$ at layer $|\word{v}| - i - 2$, and the number of children of each set.
To determine the size of the sets, two key observations must be made.
The first is that given the word $\word{u} \in \mathcal{E}(\word{v},i,x,j,\word{s})$ and the symbol $y \in \Sigma$, if $y \word{u} \in \mathcal{TEN}(\word{v},i,x)$ then there exists some pair $j' \in [n], \word{s}' \sqsubseteq_{|\word{u}| + 1} \word{v}$ such that $y \word{u} \in \mathcal{E}(\word{v},i,x,j',\word{s}')$.
Secondly, if $y \word{u} \in \mathcal{E}(\word{v},i,x,j',\word{s}')$, then $y \word{w} \in \mathcal{E}(\word{v},i,x,j',\word{s}')$ for every $\word{w} \in \mathcal{E}(\word{v},i,x,j,\word{s})$.
These observations are proven in Lemma \ref{lem:enclose_j_s_prime}, as well as showing how to determine the values of $j'$ and $\word{s}'$.

\begin{lemma}
Given $\word{u} \in \mathcal{E}(\word{v},i,x,j,\word{s})$ and symbol $y \in \Sigma$, the pair $j' \in [n], \word{s}' \sqsubseteq_{|\word{u}| + 1} \word{v}$ such that $y \word{u} \in \mathcal{E}(\word{v},i,x,j',\word{s}')$ can be computed in constant time.
Further, if $y \word{u} \in \mathcal{E}(\word{v},i,x,j',\word{s}')$, then $y \word{w} \in \mathcal{E}(\word{v},i,x,j',\word{s}')$ for every $\word{w} \in \mathcal{E}(\word{v},i,x,j,\word{s})$.
\label{lem:enclose_j_s_prime}
\end{lemma}

\begin{proof}
Assume that the array $XW$ given in Section \ref{subsec:bounding_subwords} has been precomputed.
Following the same arguments as presented in Lemma \ref{lem:PO_cartesian}, the value of $j'$ is either $j + 1$, if $y = \word{v}_{j + 1}$, or $0$ otherwise.
Similarly, the value of $\word{s}'$ is equal to the value given by $XW[\word{s},\word{w}_{1}]$.
Note that if $y \word{s}' < \word{v}$ then there is no such value of $\word{s}'$.
Similarly if $y < \word{v}_{j + 1}$ then there no value of $j'$.
To show that $y \word{w} \in \mathcal{E}(\word{v},i,x,j',\word{s}')$ for every $\word{w} \in \mathcal{E}(\word{v},i,x,j,\word{s})$, recall from Lemma \ref{lem:bounding_same_xw} that if $\word{s}'$ bounds $y \word{s}$, then $\word{s}'$ bounds $y \word{w}$ for every $\word{w}$ bounded by $\word{s}$.
Similarly, if $j'$ is the length longest suffix of $\word{u}^R y$ that is a prefix of $\word{v}$, $j'$ must also be the length of the longest suffix of $\word{w}^R y$ that is a prefix of $\word{v}$.
\end{proof}

From Lemma \ref{lem:enclose_j_s_prime}, the size of $\mathcal{E}(\word{v},i,x,j,\word{s})$ are computed using the sizes of $\mathcal{E}(\word{v},i,x,j',\word{s}')$ for $j' \in [0,n]$ and $\word{s}' \in \mathbf{S}(\word{v},|\word{s}| + 1)$.
To compute the value of $\mathcal{E}(\word{v},i,x,j,\word{s})$, an array $SE$ of size $k \times n \times n \times n^2$ is introduced such that the value of $SE[x,i,j,\word{s}] = |\mathcal{E}(\word{v},i,x,j,\word{s})|$.

\begin{lemma}
Let $\word{v} \in \Sigma^n$.
Let $SE$ be a $n \times n^2$ array such that $SE[x,i,j,\word{s}] = |\mathcal{E}(\word{v},i,x,j,\word{s})|$ for $j \in [0,n]$ and $\word{s} \sqsubseteq \word{v}$.
Every value of $SE[x,i,j,\word{s}]$ is computed in $O(k^2\cdot n^4)$ time.
\label{lem:SO_computation}
\end{lemma}

\begin{proof}
Initially the value of $SE[j,\word{s}]$ is set to 0.
Observe that every word $\word{w} \in \mathcal{E}(\word{v},i,x,j,\word{s})$ where $|\word{w}| > 1$ can be written as $z \word{w}'$ for $\word{w}' \in \mathcal{E}(\word{v},i,x,j',\word{s}')$.
From Lemma \ref{lem:enclose_j_s_prime}, the value of $j'$ and $\word{s}'$ can be calculated in constant time.
Therefore to efficiently compute the values of $SE$, it is reasonable to start by computing the size of $\mathcal{E}(\word{v},i,x,j,\word{s})$ for every $i \in [0,n], x \in \Sigma, j \in [0,n]$ and $\word{s} \in \mathbf{S}(\word{v},n-1)$.
Given $i \in [0,n], x \in \Sigma, j \in [0,n]$ and $\word{s} \in \mathbf{S}(\word{v},n-1)$, the size of $\mathcal{E}(\word{v},i,x,j,\word{s})$ is computed directly by checking each value of $z \in \Sigma$.
If $z \geq (\word{v}_{[1,i]} x)_{j + 1 \bmod i + 1}$ and $x \word{s} > \word{v}$ then the value of $SE[i,x,j,\word{s}]$ is incremented by 1, otherwise it remains the same.

Once the value of $SE[i,x,j,\word{s}]$ has been computed for every value of $i \in [1,n], x \in \Sigma, j \in [0,n]$ and $\word{s} \in \mathbf{S}(\word{v},n-1)$, the next step is to compute the value of $SE[i',x',j',\word{s}']$ for every $i' \in [1,n], x' \in \Sigma, j' \in [0,n]$ and $\word{s}' \in \mathbf{S}(\word{v},n-2)$.
This is done by looking at each value of $i \in [1,n], x \in \Sigma, j \in [0,n],\word{s} \in \mathbf{S}(\word{v},n-1)$ and $z \in \Sigma$ and determining the values of $j'$ and $\word{s}'$ for which $z \word{w} \in \mathcal{E}(\word{v},i,x,j',\word{s}')$ where $\word{w} \in \mathcal{E}(\word{v},i,x,j,\word{s})$ following Lemma \ref{lem:enclose_j_s_prime}.
Once the value of $j'$ and $\word{s}'$ has been determined, $SE[x,i,j',\word{s}']$ is increased by $SE[x,i,j,\word{s}]$.
By repeating this for every value of $i \in [1,n], x \in \Sigma, j \in [0,n],\word{s} \in \mathbf{S}(\word{v},n-1)$ and $z \in \Sigma$ leaves the value of $SE[x,i,j',\word{s}']$ as the size of $\mathcal{E}(\word{v},i,x,j',\word{s}')$.

Let $t \in [1,n - 1]$.
Once every value of $SE[x,i,j,\word{s}]$ for every value of $i \in [1,n], x \in \Sigma, j \in [0,n],$ and $\word{s} \in \mathbf{S}(\word{v},t)$, the value of $SE[x',i',j',\word{s}']$ is computed for every $i \in [1,n], x \in \Sigma, j \in [0,n],
\word{s} \in \mathbf{S}(\word{v},t - 1)$.
This is done by determining the value of $j'$ and $\word{s}'$ for each combination of $i \in [1,n], x \in \Sigma, j \in [0,n],\word{s} \in \mathbf{S}(\word{v},t)$ and $z \in \Sigma$ following Lemma \ref{lem:enclose_j_s_prime}.
Once the value of $j'$ and $\word{s}'$ has been determined, $SE[x,i,j',\word{s}']$ is increased by $SE[x,i,j,\word{s}]$.
By repeating this for every value of $i \in [1,n], x \in \Sigma, j \in [0,n],\word{s} \in \mathbf{S}(\word{v},t)$ and $z \in \Sigma$ leaves the value of $SE[x,i,j',\word{s}']$ as the size of $\mathcal{E}(\word{v},i,x,j',\word{s}')$.

Repeating this for every value of $t$ from $n - 1$ to $1$ will completely compute the array $SE$.
In order to compute this array, observe that for each of the $O(n)$ values of $t$, there are $O(n)$ values of $i,j$ and $\word{s}$ to check alongside $O(k)$ values of $x$ and $z$.
As each combination only needs to be checked once, and the process of determining $j'$ and $\word{s}'$ can be done in constant time, the total complexity is $O(n^4 \cdot k^2)$.
\end{proof}

Once $SE$ has been computed, the number of enclosing words can be computed using $SE$ and each valid combination of $i$ and $x$.
This is done in a direct manner.
Note that the number of possible values of $\word{\phi}$ such that $\word{v}_{[1,i]} x \word{\phi}$ represents a bracelet enclosing $\word{v}$ is equal to $SE[x,i,j,\word{s}]$ where $j$ is the longest suffix of $\word{v}_{[2,i]} x$ that is a prefix of $\word{v}$ and $\word{s}$ is the subword that bounds $x \word{v}_{[1,i]}^R$.
As both values can be computed naively in $O(n^2)$ operations, the complexity of this problem comes predominately from computing $SE$.

\begin{theorem}
The number of bracelets enclosing $\word{v} \in \Sigma^n$ can be computed in $O(n^4 \cdot k^2)$.
\label{thm:enclosing_complexity}
\end{theorem}

\begin{proof}
From Lemma \ref{lem:SO_computation} the array $SE$ may be computed in $O(n^4 \cdot k^2)$ operations.
Using $SE$, let $i \in [1,n]$ and $x \in \Sigma$.
Further let $l$ be the length of the longest Lyndon word that is a prefix of $\word{v}_{[1,i]}$.
If the value of $x$ is less than $\word{v}_{i + 1 \bmod l}$ or greater than or equal to $\word{v}_{i + 1}$ then there is no bracelet represented by $\word{v}_{[1,i]} x \word{\phi}$.
Similarly if $x \word{v}_{[1,i]}^R < \word{v}_{[1,i + 1]}$, then any bracelet of the form $\word{v}_{1,i} x \word{\phi}$ does not enclose $\word{v}$.
Otherwise, the number of enclosing bracelets represented by $\word{v}_{[1,i]} x \word{\phi}$ is equal to $SE[x,i,j,\word{s}']$ where $j$ is the longest suffix of $\word{v}_{[2,i]} x$ that is a prefix of $\word{v}$ and $\word{s}$ is the subword that bounds $x \word{v}_{[1,i]}^R$.
By summing the value of $SE[x,i,j,\word{s}']$ for each value of $i \in [1,n]$ and $x \in \Sigma$ such that $\word{v}_{[1,i]} x$ is the prefix of the representation of some bracelet enclosing $\word{v}$ gives the number of enclosing bracelets.
Therefore $RE(\word{v}) = \sum\limits_{i \in [1,n - 1]} \sum\limits_{x \in \Sigma} \begin{cases}
0 & x \word{v}_{[1,i]}^R < \word{v}\\
0 & x \leq \word{v}_{i + 1 \bmod l} \text{ or } x > \word{v}_{i + 1}\\
SE[x,i,j,\word{s}'] & Otherwise.
\end{cases}$
\end{proof}

\noindent
\textbf{Proof of Theorem \ref{thm:ranking_complexity}}.
The tools are now available to prove Theorem \ref{thm:ranking_complexity} and show that it is possible to rank a word $\word{v} \in \Sigma^n$ with respect to the set of bracelets of length $n$ over the alphabet $\Sigma$ in $O(k^2 \cdot n^4)$ steps.
To rank bracelets, it is sufficient to use the results of ranking $\word{v}$ with respect to necklaces, palindromic necklaces and bracelets enclosing $\word{v}$, combining them as shown in Lemma \ref{lem:ranking_bracelets}.
Sawada et. al. provided an algorithm to rank $v$ with respect to necklaces in $O(n^2)$ time.
It follows from Theorem \ref{thm:pallindromic_comlexity} that the rank with respect to palindromic necklaces can be computed in $O(k \cdot n^3)$ time.
Theorem \ref{thm:enclosing_complexity} shows that the rank with respect to bracelets enclosing $v$ can be computed in $O(k^2 \cdot n^4)$ time.
As combining these results can be done in $O(1)$ steps, therefore the overall complexity is $O(k^2 \cdot n^4)$.

\section{Conclusions and Future Work}
\label{sec:conclusions}

In this work we have presented an algorithm for the ranking of bracelets in $O(k\cdot n^4)$ time.
This expands upon the previous work on ranking necklaces and Lyndon words in $O(n^2)$ time.
Along side ranking bracelets, this work provides methods to rank palindromic necklaces in $O(k\cdot n^3)$ time, and enclosing bracelets in $O(k\cdot n^4)$ time.
There are two obvious questions to expand this work in.
The first question is if there exists a faster algorithm for ranking bracelets, which may be achieved by finding a faster algorithm to count the number of enclosing bracelets and palindromic necklaces.
The second question is if these techniques may be extend to the fixed density or fixed content cases.

The authors would like to thank the reviewers of the short version of this paper for help comments.
The authors thank the Leverhulme Trust for funding this research via the Leverhulme Research Centre for Functional Materials Design.  

\bibliography{bib.bib}

\begin{thebibliography}{10}

\bibitem{Adamson2021}
Duncan Adamson, Argyrios Deligkas, Vladimir~V. Gusev, and Igor Potapov.
\newblock Ranking bracelets in polynomial time.
\newblock {\em 32nd Annual Symposium on Combinatorial Pattern Matching}, page
  TBD, 2021.

\bibitem{Cattell2000}
K.~Cattell, F.~Ruskey, J.~Sawada, M.~Serra, and C.R. Miers.
\newblock {Fast Algorithms to Generate Necklaces, Unlabeled Necklaces, and
  Irreducible Polynomials over GF(2)}.
\newblock {\em Journal of Algorithms}, 37(2):267--282, 2000.

\bibitem{Collins17}
C.~Collins, M.~S. Dyer, M.~J. Pitcher, G.~F.~S. Whitehead, M.~Zanella,
  P.~Mandal, J.~B. Claridge, G.~R. Darling, and M.~J. Rosseinsky.
\newblock {Accelerated discovery of two crystal structure types in a complex
  inorganic phase field}.
\newblock {\em Nature}, 546(7657):280--284, 2017.

\bibitem{CircularSplicing}
Clelia {De Felice}, Rocco Zaccagnino, and Rosalba Zizza.
\newblock Unavoidable sets and circular splicing languages.
\newblock {\em Theoretical Computer Science}, 658:148 -- 158, 2017.
\newblock Formal Languages and Automata: Models, Methods and Application In
  honour of the 70th birthday of Antonio Restivo.

\bibitem{SplitNeck19}
Aris Filos{-}Ratsikas and Paul~W. Goldberg.
\newblock The complexity of splitting necklaces and bisecting ham sandwiches.
\newblock In Moses Charikar and Edith Cohen, editors, {\em Proceedings of the
  51st Annual {ACM} {SIGACT} Symposium on Theory of Computing, {STOC} 2019,
  Phoenix, AZ, USA, June 23-26, 2019}, pages 638--649. {ACM}, 2019.

\bibitem{Graham1994}
R.~L. Graham, D.~E. Knuth, and O.~Patashnik.
\newblock {\em Concrete mathematics : a foundation for computer science}.
\newblock Addison-Wesley, 1994.

\bibitem{Gupta1983}
U.~I. Gupta, D.~T. Lee, and C.~K. Wong.
\newblock {Ranking and unranking of B-trees}.
\newblock {\em Journal of Algorithms}, 4(1):51--60, mar 1983.

\bibitem{Mutze20}
Elizabeth Hartung, Hung~Phuc Hoang, Torsten M{\"{u}}tze, and Aaron Williams.
\newblock Combinatorial generation via permutation languages.
\newblock In Shuchi Chawla, editor, {\em Proceedings of the 2020 {ACM-SIAM}
  Symposium on Discrete Algorithms, {SODA} 2020, Salt Lake City, UT, USA,
  January 5-8, 2020}, pages 1214--1225. {SIAM}, 2020.

\bibitem{Karim2013}
S.~Karim, J.~Sawada, Z.~Alamgir, and S.~M. Husnine.
\newblock {Generating bracelets with fixed content}.
\newblock {\em Theoretical Computer Science}, 475:103--112, mar 2013.

\bibitem{Knuth4a}
Donald~E. Knuth.
\newblock {\em The Art of Computer Programming: Combinatorial Algorithms, Part
  1}.
\newblock Addison-Wesley Professional, 1st edition, 2011.

\bibitem{Kociumaka2014}
T.~Kociumaka, J.~Radoszewski, and W.~Rytter.
\newblock {Computing k-th Lyndon word and decoding lexicographically minimal de
  Bruijn sequence}.
\newblock In {\em Symposium on Combinatorial Pattern Matching}, pages 202--211.
  Springer, 2014.

\bibitem{Kopparty2016}
S.~Kopparty, M.~Kumar, and M.~Saks.
\newblock {Efficient indexing of necklaces and irreducible polynomials over
  finite fields}.
\newblock {\em Theory of Computing}, 12(1):1--27, 2016.

\bibitem{RankPerm1}
Martin Mare{\v{s}} and Milan Straka.
\newblock Linear-time ranking of permutations.
\newblock In Lars Arge, Michael Hoffmann, and Emo Welzl, editors, {\em
  Algorithms -- ESA 2007}, pages 187--193, Berlin, Heidelberg, 2007. Springer
  Berlin Heidelberg.

\bibitem{RankPerm2}
Wendy Myrvold and Frank Ruskey.
\newblock Ranking and unranking permutations in linear time.
\newblock {\em Information Processing Letters}, 79(6):281 -- 284, 2001.

\bibitem{Pallo1986}
J.~M. Pallo.
\newblock {Enumerating, Ranking and Unranking Binary Trees}.
\newblock {\em The Computer Journal}, 29(2):171--175, feb 1986.

\bibitem{Sawada2017}
J.~Sawada and A.~Williams.
\newblock {Practical algorithms to rank necklaces, Lyndon words, and de Bruijn
  sequences}.
\newblock {\em Journal of Discrete Algorithms}, 43:95--110, 2017.

\bibitem{Sawada2001}
Joe Sawada.
\newblock {Generating bracelets in constant amortized time}.
\newblock {\em SIAM Journal on Computing}, 31(1):259--268, jan 2001.

\bibitem{RankComb}
Toshihiro Shimizu, Takuro Fukunaga, and Hiroshi Nagamochi.
\newblock Unranking of small combinations from large sets.
\newblock {\em Journal of Discrete Algorithms}, 29:8 -- 20, 2014.

\bibitem{RankPartition}
S.~G. Williamson.
\newblock Ranking algorithms for lists of partitions.
\newblock {\em SIAM Journal on Computing}, 5(4):602--617, 1976.

\end{thebibliography}
\bibliographystyle{plain}

\end{document}